\documentclass[12pt]{article}

\usepackage[margin=1in]{geometry}
\usepackage{amsmath,amsthm,amssymb,amsfonts}
\usepackage{mathtools}
\usepackage{hyperref}
\usepackage{url}
\usepackage[numbers,sort&compress]{natbib}
\usepackage{enumitem}
\usepackage{silence}
\usepackage[expansion=false]{microtype}
\usepackage{tikz}
\usetikzlibrary{arrows.meta,positioning}

\theoremstyle{plain}
\newtheorem{theorem}{Theorem}[section]
\newtheorem{proposition}[theorem]{Proposition}
\newtheorem{lemma}[theorem]{Lemma}
\newtheorem{corollary}[theorem]{Corollary}

\theoremstyle{definition}
\newtheorem{definition}[theorem]{Definition}
\newtheorem{example}[theorem]{Example}

\theoremstyle{remark}
\newtheorem{remark}[theorem]{Remark}

\newcommand{\R}{\mathbb{R}}
\newcommand{\C}{\mathbb{C}}
\newcommand{\tr}{\operatorname{tr}}
\newcommand{\adj}{\operatorname{adj}}
\newcommand{\doi}[1]{\textsc{doi:}~\texttt{#1}}
\DeclareMathOperator{\dt}{det}
\DeclareMathOperator{\rank}{rank}
\DeclareMathOperator{\Res}{Res}
\DeclareMathOperator{\spec}{spec}
\newcommand{\ip}[2]{\langle #1,\, #2\rangle}
\newcommand{\Ec}{E^{c}}
\newcommand{\Etr}{E^{\mathrm{tr}}}
\newcommand{\Jtr}{J_{\mathrm{tr}}}
\newcommand{\Ltr}{L_{\mathrm{tr}}}
\newcommand{\chitr}{\chi_{\mathrm{tr}}}

\title{The Bogdanov--Takens normal-form coefficients in $\R^{n}$
	as directional derivatives of the characteristic invariants}

\author{%
	E.~Chan-L\'opez\thanks{%
		E-mail: \texttt{eduardo.clopez13@gmail.com}.\;
		\textsc{orcid}:~0009-0003-7712-7907.}}

\date{%
	Divisi\'on Acad\'emica de Ciencias B\'asicas,
	Universidad Ju\'arez Aut\'onoma de Tabasco, 86690 Cunduac\'an,
	Tabasco, Mexico}

\begin{document}
	\maketitle
	
	\begin{abstract}
		Let $X$ be a vector field on an open set of $\R^{n}$ with $X(p)=0$ and
		Jacobian $J=DX(p)$ of rank $n-1$ having $0$ as an eigenvalue of
		algebraic multiplicity two.  Let $q_{0}$ span $\ker J$ and let
		$e_{k}(A)$ denote the sum of the principal $k\times k$ minors of $A$,
		so that $\dt(\lambda I-A)=\sum_{k}(-1)^{k}e_{k}(A)\lambda^{n-k}$.
		Under the usual hyperbolicity assumption on the transverse block, we
		prove that the two quadratic coefficients $a,b$ of the
		Bogdanov--Takens (BT) normal form on the centre manifold are
		\[
		a=-\frac{1}{2}\,\frac{D_{q_{0}}e_{n}}{e_{n-2}},
		\qquad
		b=\frac{D_{q_{0}}e_{n-1}}{e_{n-2}}
		-\frac{e_{n-3}\,D_{q_{0}}e_{n}}{e_{n-2}^{2}},
		\]
		where $D_{q_{0}}e_{k}$ is the derivative of $x\mapsto e_{k}(DX(x))$ at
		$p$ along $q_{0}$.  The underlying spectral identity itself requires
		only invertibility of the transverse block and therefore remains valid
		without hyperbolicity.  Both formulas are the coefficients of
		$\lambda^{0}$ and
		$\lambda^{1}$ of one generating identity for the first-order spectral
		jet of $DX$ along $\ker J$; every higher coefficient is contaminated
		by the transverse block, and we prove that this dichotomy is exact.  The planar identities
		$a=-\tfrac12 D_{q_{0}}\dt$, $b=D_{q_{0}}\tr$ are the case $n=2$.  The
		proof uses only the $2$-jet of $X$ at $p$: no centre manifold, no
		normal-form transformation, no generalized eigenvector of the adjoint
		and no analytic perturbation theory occur in it.  We give the
		resulting coordinate-free nondegeneracy test and a geometric reading:
		the two identities are $a=-\tfrac12 D_{q_{0}}\delta$ and
		$b=D_{q_{0}}\tau$ for a central determinant $\delta$ and a central
		trace $\tau$ that are explicit rational functions of the principal
		minors of $DX$, so that BT nondegeneracy becomes transversality of the
		kernel line to the two hypersurfaces $\{\delta=0\}$ and
		$\{\tau=0\}$, and the topological type of the versal BT unfolding
		becomes the sign of an explicit polynomial in the entries of $DX(p)$
		and $D^{2}X(p)$, with no eigenvector computation.  A self-contained
		\textsc{Wolfram Language} script verifying every statement of the
		paper symbolically accompanies it.
	\end{abstract}
	
	
	
	\section{Introduction}\label{sec:intro}
	
	The Bogdanov--Takens bifurcation \cite{Bogdanov1975,Takens1974} is the
	generic codimension-two bifurcation of an equilibrium whose linearization
	has a nilpotent $2\times2$ block.  On the centre manifold the germ is
	smoothly equivalent to
	\begin{equation}\label{eq:BT}
		\dot u_{1}=u_{2},\qquad
		\dot u_{2}=a\,u_{1}^{2}+b\,u_{1}u_{2}+\mathcal O(|u|^{3}),
	\end{equation}
	and the bifurcation is nondegenerate when $ab\neq0$.  The standard route
	to $a$ and $b$ in $\R^{n}$ \cite{Kuznetsov2004,Kuznetsov2005} requires a
	Jordan chain $q_{0},q_{1}$ of $J=DX(p)$, the dual chain $p_{1},p_{0}$ of
	$J^{\top}$, and the second-order multilinear form $B=D^{2}X(p)$:
	\begin{equation}\label{eq:abKuz}
		a=\tfrac12\ip{p_{1}}{B(q_{0},q_{0})},
		\qquad
		b=\ip{p_{0}}{B(q_{0},q_{0})}+\ip{p_{1}}{B(q_{0},q_{1})}.
	\end{equation}
	In a companion paper \cite{Companion} we showed that in the plane these
	two numbers are the directional derivatives, along $\ker J$, of the two
	coefficients of the characteristic polynomial,
	\begin{equation}\label{eq:planar}
		a=-\tfrac12\,D_{q_{0}}\!\dt DX(p),
		\qquad
		b=D_{q_{0}}\!\tr DX(p),
	\end{equation}
	so that neither $q_{1}$, nor $p_{0}$, nor $p_{1}$, nor $B$ is needed.
	The natural question is whether \eqref{eq:planar} is an accident of
	dimension two, where the characteristic polynomial has exactly two
	coefficients and the linearization has no transverse directions.
	
	It is not.  In $\R^{n}$ the two \emph{lowest-order} coefficients of the
	characteristic polynomial play the role that the determinant and the
	trace play in the plane, up to a single normalizer, the determinant of
	the transverse block:
	\begin{equation}\label{eq:main}
		a=-\frac{1}{2}\,\frac{D_{q_{0}}e_{n}}{e_{n-2}},
		\qquad
		b=\frac{D_{q_{0}}e_{n-1}}{e_{n-2}}
		-\frac{e_{n-3}\,D_{q_{0}}e_{n}}{e_{n-2}^{2}}.
	\end{equation}
	For $n=2$ one has $e_{0}=1$ and $e_{-1}=0$, and \eqref{eq:main} collapses
	to \eqref{eq:planar}.
	
	Three features of \eqref{eq:main} deserve to be stated at the outset.
	
	\emph{(i) The proof is algebraic, not perturbative.}  One may derive
	\eqref{eq:main} by splitting the spectrum into a central and a transverse
	part, factoring $\chi_{DX(x)}$ accordingly and reducing to a centre
	manifold.  That route needs spectral separation, regularity of the
	spectral projectors, and the centre-manifold theorem, which is a large
	amount of machinery for a statement about the $2$-jet of $X$.  We prove
	\eqref{eq:main} instead from Jacobi's formula for the derivative of a
	determinant together with the block structure of the adjugate of a
	nilpotent-plus-invertible matrix.  What comes out is stronger
	(Theorem~\ref{thm:gen}): the whole first-order spectral jet of $DX$ along
	$\ker J$ is generated by one polynomial identity in $\lambda$, of which
	\eqref{eq:main} is the pair of lowest coefficients, and the same identity
	shows that no other coefficient is transversally clean
	(Proposition~\ref{prop:sharp} and Theorem~\ref{thm:sharp}); the latter
	rests on an elementary gap lemma for polynomials, Lemma~\ref{lem:gap},
	which we prove for want of a reference and which may be of independent
	interest.
	
	\emph{(ii) Hyperbolicity is used only to name the answer.}  The algebraic
	identity requires only $\dt \Jtr\neq0$; it remains valid when the
	transverse spectrum contains a purely imaginary pair or a resonance
	$\mu_{i}+\mu_{j}=0$, where the perturbative route breaks down.
	Hyperbolicity of $\Jtr$ enters exactly once, in
	Lemma~\ref{lem:identify}, to identify \eqref{eq:abKuz} with the
	coefficients of \eqref{eq:BT}.  We keep the two statements strictly
	apart.
	
	\emph{(iii) The identities have a geometric content.}  Setting
	$\delta:=e_{n}/e_{n-2}$ and $\tau:=(e_{n-1}e_{n-2}-e_{n}e_{n-3})/e_{n-2}^{2}$,
	both rational functions of the principal minors of $DX$, we show in
	Section~\ref{sec:geom} that $\delta$ and $\tau$ vanish at $p$ and that
	$a=-\tfrac12 D_{q_{0}}\delta$, $b=D_{q_{0}}\tau$.  Thus \eqref{eq:main}
	is literally \eqref{eq:planar} with the determinant and the trace
	replaced by a \emph{central} determinant and a \emph{central} trace, and
	the nondegeneracy conditions become transversality of the kernel line to
	two hypersurfaces of phase space.  This yields, in
	Corollary~\ref{cor:sign}, the topological type of the versal unfolding as
	the sign of an explicit polynomial in the entries of $DX(p)$ and
	$D^{2}X(p)$.
	
	Sections~\ref{sec:setting}--\ref{sec:main} contain the setting, the
	identification lemma and the two theorems;
	Section~\ref{sec:cons} the consequences;
	Section~\ref{sec:geom} the geometric reading;
	Section~\ref{sec:hopf} what does and does not transfer to a Hopf point;
	Section~\ref{sec:algo} the algorithm and its implementation; and
	Section~\ref{sec:examples} the examples.
	
	\section{Setting}\label{sec:setting}
	
	\subsection{Invariants and conventions}
	
	For $A\in\R^{m\times m}$ write
	\begin{equation}\label{eq:chi}
		\chi_{A}(\lambda):=\dt(\lambda I-A)
		=\sum_{k=0}^{m}(-1)^{k}e_{k}(A)\,\lambda^{m-k},
	\end{equation}
	so that $e_{k}(A)$ is the sum of the principal $k\times k$ minors of
	$A$, $e_{0}(A)=1$, $e_{1}(A)=\tr A$, $e_{m}(A)=\dt A$; equivalently
	$e_{k}(A)$ is the $k$-th elementary symmetric function of the
	eigenvalues.  We set $e_{k}:=0$ for $k<0$.
	
	For the empty matrix $A\in\R^{0\times0}$ we use throughout the standard
	conventions
	\begin{equation}\label{eq:empty}
		\chi_{A}(\lambda)=1,\qquad \dt A=1,\qquad \adj A=1,\qquad \tr A=0,
	\end{equation}
	which are the ones making \eqref{eq:chi} and the block formulas of
	Section~\ref{sec:gen} valid without case distinction.  These conventions
	are used only when $n=2$.
	
	\subsection{The standing hypothesis}
	
	\begin{definition}[Standing hypothesis]\label{def:BT0}
		Let $U\subset\R^{n}$ be open, $n\ge2$, $X\in C^{2}(U,\R^{n})$, $p\in U$
		with $X(p)=0$, and $J:=DX(p)$.  We say that $(X,p)$ satisfies
		$(\mathrm{BT}_{0})$ when
		\begin{enumerate}[label=(\roman*)]
			\item $\rank J=n-1$;
			\item $e_{n-1}(J)=0$;
			\item $e_{n-2}(J)\neq0$.
		\end{enumerate}
	\end{definition}
	
	\begin{lemma}[Equivalence, minimality, adapted splitting]\label{lem:split}
		$(X,p)$ satisfies $(\mathrm{BT}_{0})$ if and only if $0$ is an
		eigenvalue of $J$ of algebraic multiplicity two and geometric
		multiplicity one and the remaining $n-2$ eigenvalues are nonzero.  In
		that case $e_{n}(J)=0$ automatically, and
		\begin{equation}\label{eq:splitting}
			\R^{n}=\Ec\oplus\Etr,\qquad \Ec:=\ker J^{2},\quad
			\Etr:=\operatorname{im}J^{2},
		\end{equation}
		with both summands $J$-invariant, $\dim \Ec=2$, $J|_{\Ec}$ nilpotent of
		rank one, and $\Jtr:=J|_{\Etr}$ invertible.  Moreover
		\begin{equation}\label{eq:enk}
			\chi_{J}(\lambda)=\lambda^{2}\chitr(\lambda),
			\qquad
			e_{n-2}(J)=\dt\Jtr,
			\qquad
			e_{n-3}(J)=e_{n-3}(\Jtr),
		\end{equation}
		where $\chitr:=\chi_{\Jtr}$.  None of (i)--(iii) may be omitted.
	\end{lemma}
	
	\begin{proof}
		Condition (i) gives $\dt J=e_{n}(J)=0$, so $e_{n}(J)=0$ is a
		consequence and not an independent hypothesis.  By \eqref{eq:chi} the
		coefficients of $\lambda^{0},\lambda^{1},\lambda^{2}$ in $\chi_{J}$
		are $(-1)^{n}e_{n}$, $(-1)^{n-1}e_{n-1}$, $(-1)^{n-2}e_{n-2}$; hence
		$e_{n}=e_{n-1}=0$ and $e_{n-2}\neq0$ say exactly that $\lambda^{2}$
		divides $\chi_{J}$ and $\lambda^{3}$ does not, i.e.\ that $0$ is a root
		of multiplicity exactly two.  Condition (i) says $\dim\ker J=1$, i.e.\
		geometric multiplicity one.  The converse is the same computation.
		
		The Jordan structure at $0$ is therefore a single block of size two, so
		$\ker J^{2}=\ker J^{n}$ and $\operatorname{im}J^{2}=\operatorname{im}J^{n}$,
		and \eqref{eq:splitting} is Fitting's decomposition; both summands are
		$J$-invariant, $\dim\Ec=2$, $J|_{\Ec}$ is nilpotent with
		one-dimensional kernel, hence of rank one, and $\Jtr$ has no zero
		eigenvalue.  From $\chi_{J}=\chi_{J|_{\Ec}}\chitr=\lambda^{2}\chitr$,
		comparison of the coefficients of $\lambda^{2}$ and $\lambda^{3}$
		gives the two identities in \eqref{eq:enk}.
		
		Minimality.  Dropping (iii): if $J=N_{3}\oplus\Jtr$ with $N_{3}$ a
		nilpotent Jordan block of size three and $\Jtr$ invertible, then
		$\rank J=n-1$ and $e_{n-1}=0$ but $e_{n-2}=0$, and $0$ has algebraic
		multiplicity three.  Dropping (i): if $J=0_{2\times2}\oplus\Jtr$, then
		$e_{n-1}=e_{n}=0$ and $e_{n-2}=\dt\Jtr\neq0$ but $\rank J=n-2$, there
		is no Jordan chain and $q_{1}$ below does not exist.  Dropping (ii):
		if $J$ is invertible then $\ker J=0$.
	\end{proof}
	
	\begin{remark}
		For $n=2$, (i)--(iii) read $\rank J=1$, $\tr J=0$, $e_{0}=1\neq0$, i.e.\
		$J$ nilpotent of rank one, and \eqref{eq:enk} reads
		$\chi_{J}=\lambda^{2}$, $e_{0}=1=\dt\Jtr$ and $e_{-1}=0$, with the
		conventions \eqref{eq:empty}.
	\end{remark}
	
	\subsection{The coefficients \texorpdfstring{$a$ and $b$}{a and b}}
	
	\begin{definition}[Adapted chain and the pair $(a,b)$]\label{def:ab}
		Under $(\mathrm{BT}_{0})$, let $q_{0}$ span $\ker J$ and choose
		$q_{1}\in\Ec$ with $Jq_{1}=q_{0}$.  Let $p_{0},p_{1}$ be the covectors
		determined by
		\begin{equation}\label{eq:dual}
			\ip{p_{0}}{q_{0}}=\ip{p_{1}}{q_{1}}=1,\quad
			\ip{p_{0}}{q_{1}}=\ip{p_{1}}{q_{0}}=0,\quad
			p_{0}|_{\Etr}=p_{1}|_{\Etr}=0 .
		\end{equation}
		With $B(u,v):=D^{2}X(p)[u,v]$, the BT coefficients associated with
		$(X,p,q_{0})$ are defined by \eqref{eq:abKuz}.
	\end{definition}
	
	\begin{lemma}\label{lem:welldef}
		Under $(\mathrm{BT}_{0})$, $q_{1}$ exists, $(p_{0},p_{1})$ is uniquely
		determined by \eqref{eq:dual} and satisfies $J^{\top}p_{1}=0$,
		$J^{\top}p_{0}=p_{1}$.  Moreover:
		\begin{enumerate}[label=(\alph*)]
			\item the pair $(a,b)$ does not depend on the choice of $q_{1}$;
			\item replacing $q_{0}$ by $cq_{0}$, $c\neq0$, forces
			$q_{1}\mapsto cq_{1}$, $p_{i}\mapsto c^{-1}p_{i}$, and produces
			$(a,b)\mapsto(ca,cb)$.
		\end{enumerate}
	\end{lemma}
	
	\begin{proof}
		Since $J|_{\Ec}$ is nilpotent of rank one, $\operatorname{im}J|_{\Ec}
		=\ker J|_{\Ec}=\langle q_{0}\rangle$, so $q_{1}$ exists and
		$(q_{0},q_{1})$ is a basis of $\Ec$; together with any basis of $\Etr$
		it is a basis of $\R^{n}$, and \eqref{eq:dual} says that
		$(p_{0},p_{1})$ are the first two vectors of the dual basis.  For any
		$v$, $\ip{J^{\top}p_{1}}{v}=\ip{p_{1}}{Jv}$, which vanishes for
		$v=q_{0}$ ($Jq_{0}=0$), for $v=q_{1}$ ($Jq_{1}=q_{0}\perp p_{1}$) and
		for $v\in\Etr$ ($Jv\in\Etr$); hence $J^{\top}p_{1}=0$.  The same
		computation gives $\ip{J^{\top}p_{0}}{q_{0}}=0=\ip{p_{1}}{q_{0}}$,
		$\ip{J^{\top}p_{0}}{q_{1}}=\ip{p_{0}}{q_{0}}=1=\ip{p_{1}}{q_{1}}$ and
		$\ip{J^{\top}p_{0}}{v}=0=\ip{p_{1}}{v}$ for $v\in\Etr$, i.e.\
		$J^{\top}p_{0}=p_{1}$.
		
		(a) The only freedom in $q_{1}$ is $q_{1}\mapsto q_{1}+dq_{0}$,
		$d\in\R$, and the dual basis then changes by $p_{0}\mapsto p_{0}-dp_{1}$,
		$p_{1}\mapsto p_{1}$.  In \eqref{eq:abKuz}, $a$ involves neither
		$q_{1}$ nor $p_{0}$ and is unchanged, while
		\[
		b\longmapsto
		\ip{p_{0}-dp_{1}}{B(q_{0},q_{0})}+\ip{p_{1}}{B(q_{0},q_{1}+dq_{0})}
		=b-d\ip{p_{1}}{B(q_{0},q_{0})}+d\ip{p_{1}}{B(q_{0},q_{0})}=b .
		\]
		Equivalently, in the notation of Section~\ref{sec:gen}, the diagonal
		entries of the central block change by $[L_{c}]_{11}\mapsto
		[L_{c}]_{11}-2ad$ and $[L_{c}]_{22}\mapsto[L_{c}]_{22}+2ad$, so
		$b=\tr L_{c}$ is invariant.
		
		(b) $Jq_{1}'=cq_{0}$ forces $q_{1}'\in cq_{1}+\ker J$, and by (a) we
		may take $q_{1}'=cq_{1}$; the conditions \eqref{eq:dual} then give
		$p_{0}\mapsto c^{-1}p_{0}$ and $p_{1}\mapsto c^{-1}p_{1}$, since the
		dual frame of a rescaled frame is rescaled by the inverse factor.
		Both expressions in \eqref{eq:abKuz} are homogeneous of degree two in
		the $q$'s and of degree $-1$ in the $p$'s; the two degrees do not
		cancel but combine, and the total degree in $c$ is $2-1=1$.  Hence
		$(a,b)\mapsto(ca,cb)$, and not $(c^{2}a,c^{2}b)$.
	\end{proof}
	
	\begin{remark}\label{rem:scale}
		Consequently $a$ and $b$ are attached to the pair $(X,q_{0})$ and not
		to $X$ alone, and two authors who normalize $q_{0}$ differently will
		report values differing by the corresponding factor $c$.  What is
		independent of the normalization is the vanishing of $a$, the
		vanishing of $b$, and the sign of $ab$, since $ab\mapsto c^{2}ab$ with
		$c^{2}>0$.  Example~\ref{ex:diasmello} exhibits this against a
		published computation.
	\end{remark}
	
	Formula \eqref{eq:main} is homogeneous of degree one in $q_{0}$, hence
	compatible with Lemma~\ref{lem:welldef}(b).
	
	\subsection{Identification with the normal-form coefficients}
	
	Definition~\ref{def:ab} is algebraic and involves only the $2$-jet of
	$X$.  The next lemma is the only place in the paper where hyperbolicity
	of $\Jtr$, a centre manifold, and $C^{3}$ regularity are used; they are
	needed to \emph{name} $a$ and $b$, not to prove
	Theorems~\ref{thm:gen} and \ref{thm:main}.
	
	\begin{lemma}[Identification]\label{lem:identify}
		Let $(X,p)$ satisfy $(\mathrm{BT}_{0})$, let $X$ be of class $C^{3}$
		and let $\Jtr$ be hyperbolic.  Let $W^{c}$ be a local centre manifold
		at $p$.  Then in suitable local coordinates $u=(u_{1},u_{2})$ on
		$W^{c}$, obtained from the chart induced by the basis $(q_{0},q_{1})$
		by a near-identity quadratic change, the restriction of $X$ to $W^{c}$
		is \eqref{eq:BT} with
		\[
		a=\tfrac12\ip{p_{1}}{B(q_{0},q_{0})},\qquad
		b=\ip{p_{0}}{B(q_{0},q_{0})}+\ip{p_{1}}{B(q_{0},q_{1})} ,
		\]
		that is, with the numbers of Definition~\ref{def:ab}.
	\end{lemma}
	
	\begin{proof}
		\emph{Step 1: the quadratic part of the reduced field.}  Since $\Jtr$
		is hyperbolic and $X\in C^{3}$, a local centre manifold exists and is
		of class $C^{3}$; it is the graph
		$W^{c}=\{p+y+h(y): y\in\Ec,\ |y|<\varepsilon\}$ of a map
		$h:\Ec\to\Etr$ with $h(0)=0$ and $Dh(0)=0$, so $h(y)=\mathcal O(|y|^{2})$.
		Let $\Pi_{c}$ be the projection onto $\Ec$ along $\Etr$; the reduced
		field in the coordinate $y$ is $\dot y=\Pi_{c}X(p+y+h(y))$.  Writing
		$u=y+h(y)$ and using $X(p+u)=Ju+\tfrac12B(u,u)+\mathcal O(|u|^{3})$,
		\[
		\Pi_{c}Ju=\Pi_{c}Jy+\Pi_{c}Jh(y)=Jy ,
		\]
		because $\Ec$ is $J$-invariant and $Jh(y)\in\Etr$ ($\Etr$ is
		$J$-invariant), and
		\[
		B(u,u)=B(y,y)+2B(y,h(y))+B(h(y),h(y))=B(y,y)+\mathcal O(|y|^{3}) ,
		\]
		because $B$ is bounded bilinear and $h(y)=\mathcal O(|y|^{2})$.  Hence
		\begin{equation}\label{eq:reduced}
			\dot y=J|_{\Ec}\,y+\tfrac12\Pi_{c}B(y,y)+\mathcal O(|y|^{3}).
		\end{equation}
		In particular the quadratic part of the reduced field does not involve
		$h$; the centre manifold is not unique, but its $2$-jet is, and in any
		case only $Dh(0)=0$ was used.
		
		Write $y=y_{1}q_{0}+y_{2}q_{1}$ and pair \eqref{eq:reduced} with
		$p_{0}$ and $p_{1}$.  Since $Jq_{0}=0$ and $Jq_{1}=q_{0}$ we have
		$J|_{\Ec}y=y_{2}q_{0}$, so
		\begin{equation}\label{eq:planarform}
			\begin{aligned}
				\dot y_{1}&=y_{2}
				+\alpha_{20}y_{1}^{2}+\alpha_{11}y_{1}y_{2}+\alpha_{02}y_{2}^{2}
				+\mathcal O(|y|^{3}),\\
				\dot y_{2}&=\ \ \ \ \ \,
				\beta_{20}y_{1}^{2}+\beta_{11}y_{1}y_{2}+\beta_{02}y_{2}^{2}
				+\mathcal O(|y|^{3}),
			\end{aligned}
		\end{equation}
		with
		\begin{equation}\label{eq:albe}
			\begin{gathered}
				\alpha_{20}=\tfrac12\ip{p_{0}}{B(q_{0},q_{0})},\quad
				\alpha_{11}=\ip{p_{0}}{B(q_{0},q_{1})},\quad
				\alpha_{02}=\tfrac12\ip{p_{0}}{B(q_{1},q_{1})},\\
				\beta_{20}=\tfrac12\ip{p_{1}}{B(q_{0},q_{0})},\quad
				\beta_{11}=\ip{p_{1}}{B(q_{0},q_{1})},\quad
				\beta_{02}=\tfrac12\ip{p_{1}}{B(q_{1},q_{1})}.
			\end{gathered}
		\end{equation}
		
		\emph{Step 2: the quadratic normal form of \eqref{eq:planarform}.}  Let
		$N=\left(\begin{smallmatrix}0&1\\0&0\end{smallmatrix}\right)$ be the
		linear part of \eqref{eq:planarform} and let
		$L_{N}(h)(x):=Dh(x)Nx-Nh(x)$ be the homological operator on quadratic
		vector fields.  Under the near-identity change $y=x+h(x)$ with $h$
		quadratic, the quadratic part $F_{2}$ of \eqref{eq:planarform} becomes
		$F_{2}-L_{N}(h)+\mathcal O(|x|^{3})$.  Writing
		\[
		h_{1}=p_{1}'x_{1}^{2}+p_{2}'x_{1}x_{2}+p_{3}'x_{2}^{2},
		\qquad
		h_{2}=q_{1}'x_{1}^{2}+q_{2}'x_{1}x_{2}+q_{3}'x_{2}^{2},
		\]
		and using $Nx=(x_{2},0)$, one computes
		\[
		L_{N}(h)=\bigl(x_{2}\partial_{1}h_{1}-h_{2},\ x_{2}\partial_{1}h_{2}\bigr)
		=\Bigl(-q_{1}'x_{1}^{2}+(2p_{1}'-q_{2}')x_{1}x_{2}+(p_{2}'-q_{3}')x_{2}^{2},
		\ \ 2q_{1}'x_{1}x_{2}+q_{2}'x_{2}^{2}\Bigr).
		\]
		The transformed quadratic coefficients are therefore
		\[
		\begin{array}{lll}
			x_{1}^{2}:&\alpha_{20}+q_{1}', &\beta_{20},\\
			x_{1}x_{2}:&\alpha_{11}-2p_{1}'+q_{2}', &\beta_{11}-2q_{1}',\\
			x_{2}^{2}:&\alpha_{02}-p_{2}'+q_{3}', &\beta_{02}-q_{2}' .
		\end{array}
		\]
		Choosing
		\[
		q_{1}'=-\alpha_{20},\qquad
		q_{2}'=\beta_{02},\qquad
		p_{1}'=\tfrac12(\alpha_{11}+\beta_{02}),\qquad
		p_{2}'=\alpha_{02}+q_{3}',
		\]
		with $q_{3}',p_{3}'$ arbitrary, annihilates all three coefficients of
		the first component and the $x_{2}^{2}$ coefficient of the second, and
		leaves
		\[
		\dot x_{1}=x_{2}+\mathcal O(|x|^{3}),\qquad
		\dot x_{2}=\beta_{20}x_{1}^{2}+(2\alpha_{20}+\beta_{11})x_{1}x_{2}
		+\mathcal O(|x|^{3}).
		\]
		Hence $a_{\mathrm{NF}}=\beta_{20}$ and
		$b_{\mathrm{NF}}=2\alpha_{20}+\beta_{11}$.  Substituting
		\eqref{eq:albe} gives exactly \eqref{eq:abKuz}.
	\end{proof}
	
	\begin{remark}[Regularity]\label{rem:reg}
		The two sides of \eqref{eq:main} depend only on the $2$-jet of $X$ at
		$p$, so $X\in C^{2}$ suffices for
		Theorems~\ref{thm:gen} and \ref{thm:main} and no remainder term
		appears there.  Class $C^{3}$ is used only in
		Lemma~\ref{lem:identify}: for a merely $C^{2}$ field Taylor's theorem
		gives a remainder $o(|u|^{2})$ rather than $\mathcal O(|u|^{3})$, and
		the local centre manifold need not be $C^{3}$.  Accordingly all
		statements about \eqref{eq:BT} are made under $X\in C^{3}$, and all
		statements about the spectral invariants under $X\in C^{2}$.
	\end{remark}
	
	\section{The generating identity}\label{sec:gen}
	
	Throughout this section $(X,p)$ satisfies $(\mathrm{BT}_{0})$ with
	$X\in C^{2}$, $q_{0}$ spans $\ker J$, and
	\begin{equation}\label{eq:L}
		L:=D(DX)_{p}[q_{0}]
		=\left.\frac{d}{ds}\right|_{s=0}DX(p+sq_{0}),
		\qquad Lw=B(q_{0},w).
	\end{equation}
	We write $D_{q_{0}}f:=\frac{d}{ds}\big|_{0}f(p+sq_{0})$, and
	$L_{c}:=\Pi_{c}L|_{\Ec}$, $\Ltr:=\Pi_{\mathrm{tr}}L|_{\Etr}$ for the two
	diagonal blocks of $L$ in the splitting \eqref{eq:splitting}, which are
	intrinsically defined endomorphisms of $\Ec$ and $\Etr$.
	
	\begin{lemma}[Jacobi]\label{lem:jacobi}
		For $A,L\in\R^{m\times m}$,
		$\frac{d}{dt}\big|_{0}\dt(A+tL)=\tr(\adj(A)L)$, and consequently
		\begin{equation}\label{eq:dchi}
			\left.\frac{d}{dt}\right|_{0}\chi_{A+tL}(\lambda)
			=-\tr\bigl(\adj(\lambda I-A)\,L\bigr).
		\end{equation}
	\end{lemma}
	
	\begin{proof}
		The determinant is multilinear in the columns, so
		$\frac{d}{dt}\big|_{0}\dt(A+tL)=\sum_{j}\dt(A^{(j)})$, where $A^{(j)}$
		is $A$ with its $j$-th column replaced by the $j$-th column of $L$.
		Expanding $\dt(A^{(j)})$ along that column gives
		$\sum_{i}L_{ij}C_{ij}$ with $C_{ij}$ the $(i,j)$ cofactor of $A$, and
		since $\adj(A)_{ji}=C_{ij}$ this is $(\adj(A)L)_{jj}$.  Summing over
		$j$ gives $\tr(\adj(A)L)$.  For the second identity,
		$\lambda I-(A+tL)=(\lambda I-A)-tL$, so the derivative equals
		$\tr(\adj(\lambda I-A)(-L))$.
	\end{proof}
	
	\begin{lemma}[Block adjugate]\label{lem:blockadj}
		If $A=\operatorname{diag}(A_{1},A_{2})$ with $A_{i}$ square, then
		\[
		\adj(A)=\operatorname{diag}\bigl(\dt(A_{2})\adj(A_{1}),\
		\dt(A_{1})\adj(A_{2})\bigr).
		\]
	\end{lemma}
	
	\begin{proof}
		Both sides are polynomial in the entries; for $A_{1},A_{2}$ invertible,
		$\adj(A)=\dt(A)A^{-1}
		=\dt(A_{1})\dt(A_{2})\operatorname{diag}(A_{1}^{-1},A_{2}^{-1})$, which
		is the right-hand side.  The invertible pairs form a Zariski-dense
		subset, so the polynomial identity holds everywhere.  With
		\eqref{eq:empty} the statement is also correct when one block is empty.
	\end{proof}
	
	\begin{theorem}[Generating identity]\label{thm:gen}
		Under $(\mathrm{BT}_{0})$ with $X\in C^{2}$, for every $\lambda$,
		\begin{equation}\label{eq:gen}
			\sum_{k=0}^{n}(-1)^{k}\bigl(D_{q_{0}}e_{k}\bigr)\lambda^{n-k}
			=-\chitr(\lambda)\,\bigl(b\lambda+2a\bigr)
			-\lambda^{2}\,
			\tr\bigl(\adj(\lambda I-\Jtr)\,\Ltr\bigr),
		\end{equation}
		where $D_{q_{0}}e_{k}$ abbreviates
		$\frac{d}{ds}\big|_{0}e_{k}\bigl(DX(p+sq_{0})\bigr)$.  For $n=2$ the
		last term is $0$ by \eqref{eq:empty}.
	\end{theorem}
	
	\begin{proof}
		Differentiating \eqref{eq:chi} with $A=DX(p+sq_{0})$ at $s=0$ and using
		the chain rule together with \eqref{eq:dchi} and \eqref{eq:L},
		\begin{equation}\label{eq:step1}
			\sum_{k=0}^{n}(-1)^{k}\bigl(D_{q_{0}}e_{k}\bigr)\lambda^{n-k}
			=-\tr\bigl(\adj(\lambda I-J)\,L\bigr).
		\end{equation}
		Both sides of \eqref{eq:step1} are independent of the basis used to
		compute them, since $\adj(P^{-1}MP)=P^{-1}\adj(M)P$ and the trace of a
		product of endomorphisms is intrinsic.  With respect to
		\eqref{eq:splitting},
		\[
		J=\begin{pmatrix}N&0\\0&\Jtr\end{pmatrix},
		\qquad
		N=[J|_{\Ec}]_{(q_{0},q_{1})}
		=\begin{pmatrix}0&1\\0&0\end{pmatrix},
		\]
		the matrix of $N$ being read off from $Jq_{0}=0$ (first column zero)
		and $Jq_{1}=q_{0}$ (second column $(1,0)^{\top}$).  Since
		$\dt(\lambda I_{2}-N)=\lambda^{2}$, Lemma~\ref{lem:blockadj} gives
		\[
		\adj(\lambda I-J)
		=\begin{pmatrix}
			\chitr(\lambda)\,\adj(\lambda I_{2}-N)&0\\[2pt]
			0&\lambda^{2}\adj(\lambda I-\Jtr)
		\end{pmatrix}.
		\]
		The adjugate being block diagonal, only the diagonal blocks of $L$
		contribute to the trace in \eqref{eq:step1}:
		\begin{equation}\label{eq:step2}
			\tr\bigl(\adj(\lambda I-J)L\bigr)
			=\chitr(\lambda)\,
			\tr\bigl(\adj(\lambda I_{2}-N)\,L_{c}\bigr)
			+\lambda^{2}\tr\bigl(\adj(\lambda I-\Jtr)\,\Ltr\bigr).
		\end{equation}
		Finally
		$\adj\!\begin{psmallmatrix}\lambda&-1\\0&\lambda\end{psmallmatrix}
		=\begin{psmallmatrix}\lambda&1\\0&\lambda\end{psmallmatrix}$, and by
		\eqref{eq:L} and \eqref{eq:dual} the matrix of $L_{c}$ in the basis
		$(q_{0},q_{1})$ is
		\[
		[L_{c}]=\begin{pmatrix}
			\ip{p_{0}}{B(q_{0},q_{0})} & \ip{p_{0}}{B(q_{0},q_{1})}\\[2pt]
			\ip{p_{1}}{B(q_{0},q_{0})} & \ip{p_{1}}{B(q_{0},q_{1})}
		\end{pmatrix},
		\]
		whence, by \eqref{eq:abKuz},
		\begin{equation}\label{eq:step3}
			\tr\!\left(\begin{pmatrix}\lambda&1\\0&\lambda\end{pmatrix}
			[L_{c}]\right)
			=\lambda\bigl([L_{c}]_{11}+[L_{c}]_{22}\bigr)+[L_{c}]_{21}
			=b\lambda+2a.
		\end{equation}
		Substituting \eqref{eq:step3} into \eqref{eq:step2} and the result into
		\eqref{eq:step1} gives \eqref{eq:gen}.
	\end{proof}
	
	\begin{remark}[Degrees]\label{rem:degrees}
		For $n\ge3$, $\deg\chitr=n-2$ and
		$\deg\adj(\lambda I-\Jtr)=n-3$, so both terms on the right of
		\eqref{eq:gen} have degree $n-1$; on the left the term $k=0$ is
		$D_{q_{0}}e_{0}=0$, so the left-hand side also has degree at most
		$n-1$.  For $n=2$ the conventions \eqref{eq:empty} give
		$\chitr\equiv1$, $\Ltr$ empty and last term $0$, so \eqref{eq:gen}
		reads $-D_{q_{0}}e_{1}\lambda+D_{q_{0}}e_{2}=-(b\lambda+2a)$, i.e.
		\[
		D_{q_{0}}e_{1}=b,\qquad D_{q_{0}}e_{2}=-2a,
		\]
		which are the planar identities \eqref{eq:planar}.
	\end{remark}
	
	\section{The main theorem}\label{sec:main}
	
	\begin{theorem}\label{thm:main}
		Under $(\mathrm{BT}_{0})$ with $X\in C^{2}$, and with all invariants
		evaluated at $p$,
		\begin{equation}\label{eq:two}
			D_{q_{0}}e_{n}=-2a\,e_{n-2},
			\qquad
			D_{q_{0}}e_{n-1}=b\,e_{n-2}-2a\,e_{n-3},
		\end{equation}
		and therefore
		\begin{equation}\label{eq:mainbis}
			a=-\frac{1}{2}\,\frac{D_{q_{0}}e_{n}}{e_{n-2}},
			\qquad
			b=\frac{D_{q_{0}}e_{n-1}}{e_{n-2}}
			-\frac{e_{n-3}\,D_{q_{0}}e_{n}}{e_{n-2}^{2}}.
		\end{equation}
		If in addition $X\in C^{3}$ and $\Jtr$ is hyperbolic, then by
		Lemma~\ref{lem:identify} the numbers \eqref{eq:mainbis} are the
		quadratic coefficients of the Bogdanov--Takens normal form
		\eqref{eq:BT} of the restriction of $X$ to a centre manifold at $p$.
	\end{theorem}
	
	\begin{proof}
		Compare the coefficients of $\lambda^{0}$ and $\lambda^{1}$ in
		\eqref{eq:gen}; the last term of \eqref{eq:gen} carries a factor
		$\lambda^{2}$ and contributes to neither.  Writing
		$\chitr(\lambda)=\sum_{j=0}^{n-2}(-1)^{j}e_{j}(\Jtr)\lambda^{n-2-j}$
		and using \eqref{eq:enk}, the coefficients of $\lambda^{0}$ and
		$\lambda^{1}$ in $\chitr$ are $(-1)^{n-2}e_{n-2}(\Jtr)=(-1)^{n}e_{n-2}$
		and $(-1)^{n-3}e_{n-3}(\Jtr)=(-1)^{n-1}e_{n-3}$.  Hence
		\[
		(-1)^{n}D_{q_{0}}e_{n}=-(-1)^{n}e_{n-2}\cdot2a,
		\qquad
		(-1)^{n-1}D_{q_{0}}e_{n-1}
		=-\bigl[(-1)^{n}e_{n-2}\,b+(-1)^{n-1}e_{n-3}\cdot2a\bigr],
		\]
		and dividing by $(-1)^{n}$ and $(-1)^{n-1}$ respectively gives
		\eqref{eq:two}.  Since $e_{n-2}\neq0$, \eqref{eq:two} can be solved for
		$a$ and then for $b$, giving \eqref{eq:mainbis}.  The signs depending
		on $n$ cancel identically.
	\end{proof}
	
	\begin{corollary}[The planar case]\label{cor:planar}
		For $n=2$, $e_{0}=1$ and $e_{-1}=0$, and \eqref{eq:mainbis} reduces to
		\eqref{eq:planar}.
	\end{corollary}
	
	\section{Consequences}\label{sec:cons}
	
	\subsection{A coordinate-free nondegeneracy test}
	
	\begin{corollary}\label{cor:test}
		Let $X\in C^{3}$ with $X(p)=0$, $J=DX(p)$, and let $q_{0}$ be any
		nonzero element of $\ker J$.  Then $p$ is a nondegenerate BT
		singularity of $X$, meaning that $(\mathrm{BT}_{0})$ holds, $\Jtr$ is
		hyperbolic, and the restriction of $X$ to a centre manifold is
		smoothly equivalent to \eqref{eq:BT} with $ab\neq0$, if and only if
		\begin{enumerate}[label=(\alph*)]
			\item $\rank J=n-1$, $e_{n-1}(J)=0$, $e_{n-2}(J)\neq0$;
			\item all roots of $\chi_{J}(\lambda)/\lambda^{2}$ have nonzero real
			part;
			\item $D_{q_{0}}e_{n}\neq0$;
			\item $e_{n-2}\,D_{q_{0}}e_{n-1}\neq e_{n-3}\,D_{q_{0}}e_{n}$.
		\end{enumerate}
		Conditions (c) and (d) do not depend on the normalization of $q_{0}$.
	\end{corollary}
	
	\begin{proof}
		(a) is Definition~\ref{def:BT0}, (b) is hyperbolicity of $\Jtr$ by
		\eqref{eq:enk}, and (c), (d) are $a\neq0$ and $b\neq0$ after clearing
		denominators in \eqref{eq:mainbis}.  Both are homogeneous of degree one
		in $q_{0}$, so their nonvanishing is scale-invariant.
	\end{proof}
	
	\begin{remark}
		Corollary~\ref{cor:test} is a statement about the germ of a single
		vector field.  The genericity of a two-parameter \emph{family} through
		$p$ requires in addition the transversality of the unfolding
		\cite[Thm.~8.4]{Kuznetsov2004}, which is a condition on the parameter
		dependence and is not addressed here.
	\end{remark}
	
	Every ingredient of Corollary~\ref{cor:test} is a polynomial in the
	entries of the $2$-jet of $X$ at $p$, except for the one-dimensional
	kernel computation, so the test is directly amenable to symbolic
	elimination in the parameters.
	
	\subsection{Invariance under changes of phase-space coordinates}
	
	\begin{lemma}[Invariance of the first-order spectral jet along
		$\ker J$]\label{lem:inv}
		Let $X\in C^{2}$ with $X(p)=0$, let $q_{0}\in\ker DX(p)$, let
		$\varphi$ be a $C^{2}$ diffeomorphism with $\varphi(z_{0})=p$, put
		$Y:=(D\varphi)^{-1}(X\circ\varphi)$ and
		$q_{0}^{z}:=D\varphi(z_{0})^{-1}q_{0}$.  Then, for every
		$k=0,\dots,n$,
		\[
		D_{q_{0}^{z}}\,e_{k}\bigl(DY\bigr)(z_{0})
		=D_{q_{0}}\,e_{k}\bigl(DX\bigr)(p).
		\]
	\end{lemma}
	
	\begin{proof}
		Write $DY(z)=M(z)+R(z)$ with
		$M(z)=D\varphi(z)^{-1}DX(\varphi(z))D\varphi(z)$ and
		$R(z)\zeta=\bigl[D_{\zeta}(D\varphi(z)^{-1})\bigr]X(\varphi(z))$.
		Since $M(z)$ is similar to $DX(\varphi(z))$ we have
		$e_{k}(M(z))=e_{k}(DX(\varphi(z)))$ for all $k$, and the chain rule
		together with $D\varphi(z_{0})q_{0}^{z}=q_{0}$ gives
		\[
		\left.\frac{d}{dt}\right|_{0}e_{k}\bigl(M(z_{0}+tq_{0}^{z})\bigr)
		=D_{q_{0}}e_{k}(DX)(p).
		\]
		It remains to show that $R$ contributes nothing.  We do \emph{not}
		differentiate $R$, which would require a third derivative of
		$\varphi$; we estimate it.  Since $\varphi\in C^{2}$, the map
		$z\mapsto D(D\varphi(z)^{-1})$ is continuous, hence bounded near
		$z_{0}$, so
		\[
		\|R(z_{0}+tq_{0}^{z})\|\le C\,\bigl\|X\bigl(\varphi(z_{0}+tq_{0}^{z})\bigr)\bigr\| .
		\]
		Now $\varphi(z_{0}+tq_{0}^{z})=p+tq_{0}+\mathcal O(t^{2})$, and since
		$X\in C^{1}$ with $X(p)=0$,
		\[
		X\bigl(\varphi(z_{0}+tq_{0}^{z})\bigr)
		=DX(p)\bigl(tq_{0}+\mathcal O(t^{2})\bigr)+o(t)
		=t\,DX(p)q_{0}+\mathcal O(t^{2})+o(t)=o(t),
		\]
		the linear term vanishing because $q_{0}\in\ker DX(p)$.  Hence
		$\|R(z_{0}+tq_{0}^{z})\|=o(t)$.  Finally $e_{k}$ is a polynomial map,
		so $e_{k}(M+R)-e_{k}(M)=\mathcal O(\|R\|)$ locally uniformly, and
		\[
		e_{k}\bigl(DY(z_{0}+tq_{0}^{z})\bigr)
		-e_{k}\bigl(M(z_{0}+tq_{0}^{z})\bigr)=o(t),
		\]
		so the two derivatives at $t=0$ coincide.
	\end{proof}
	
	The proof shows that the two hypotheses $X(p)=0$ and $q_{0}\in\ker DX(p)$
	play distinct and equally necessary roles: the first makes $R(z_{0})=0$,
	the second kills the linear term of $X\circ\varphi$ along $q_{0}^{z}$.
	Neither alone suffices.
	
	\begin{remark}[The kernel hypothesis is not cosmetic]\label{rem:necessary}
		Take $n=2$, $X(u,v)=(v+u^{2},\,3u^{2}-uv)$, which satisfies
		$(\mathrm{BT}_{0})$ at the origin with $\ker J=\langle(1,0)\rangle$,
		and $\varphi(u,v)=(u+uv,\,v+u^{2})$, a near-identity $C^{\infty}$
		diffeomorphism.  Then $D_{q}e_{2}$ at the origin equals $-6$ for
		$q=(1,0)\in\ker J$ both for $X$ and for
		$Y=(D\varphi)^{-1}(X\circ\varphi)$, whereas for $q=(0,1)\notin\ker J$
		it equals $1$ for $X$ and $3$ for $Y$.  Away from $\ker J$ the
		functions $x\mapsto e_{k}(DX(x))$ are simply not invariant.
	\end{remark}
	
	Together with Lemma~\ref{lem:welldef}(b), Lemma~\ref{lem:inv} shows that
	both sides of \eqref{eq:mainbis} transform in the same way, so
	\eqref{eq:mainbis} may be evaluated in whichever coordinates the model
	happens to be written.
	
	\subsection{Sharpness: which invariants are transversally clean}
	
	The generating identity determines the whole first-order jet, and shows
	where the transverse block enters.  Write
	$\chitr(\lambda)=\sum_{j=0}^{n-2}c_{j}\lambda^{j}$,
	\[
	\adj(\lambda I-\Jtr)=\sum_{j=0}^{\nu-1}A_{j}\lambda^{j},
	\qquad \nu:=n-2,
	\qquad
	t_{j}:=\tr(A_{j}\Ltr),
	\]
	with the conventions $c_{j}:=0$ for $j<0$ or $j>n-2$, and $A_{j}:=0$,
	hence $t_{j}:=0$, for $j<0$ or $j>\nu-1$.
	
	\begin{proposition}\label{prop:sharp}
		Under $(\mathrm{BT}_{0})$ with $n\ge3$ and $X\in C^{2}$, for
		$0\le\rho\le n-1$,
		\begin{equation}\label{eq:allcoeff}
			(-1)^{n-\rho}D_{q_{0}}e_{n-\rho}
			=-\bigl(2a\,c_{\rho}+b\,c_{\rho-1}\bigr)-t_{\rho-2},
		\end{equation}
		with the conventions above.  The cases $\rho=0,1$ are \eqref{eq:two}:
		$D_{q_{0}}e_{n}$ and $D_{q_{0}}e_{n-1}$ are determined by $J$ and
		$(a,b)$ alone.  The case $\rho=n-1$ is the explicit identity
		\begin{equation}\label{eq:trace}
			D_{q_{0}}e_{1}=b+\tr\Ltr .
		\end{equation}
		Moreover $\Ltr$ is an arbitrary endomorphism of $\Etr$: given $J$ and
		any linear $L$ there is a $C^{\infty}$ vector field $X$ with
		$DX(p)=J$ and $D(DX)_{p}[q_{0}]=L$.
	\end{proposition}
	
	\begin{proof}
		\eqref{eq:allcoeff} is the extraction of the coefficient of
		$\lambda^{\rho}$ in \eqref{eq:gen}.  For \eqref{eq:trace} it is quicker
		to argue directly:
		$D_{q_{0}}e_{1}=D_{q_{0}}\tr DX=\tr L=\tr L_{c}+\tr\Ltr$, and
		$\tr L_{c}=b$ by \eqref{eq:step3}.  Realizability of $L$: choose
		coordinates with $q_{0}=\mathrm{e}_{1}$ and let $X$ be the quadratic
		field $X(x)=J(x-p)+\tfrac12 B(x-p,x-p)$ with $B$ the symmetric
		bilinear form defined by $B(\mathrm{e}_{1},\mathrm{e}_{j})
		=B(\mathrm{e}_{j},\mathrm{e}_{1}):=L\mathrm{e}_{j}$ and all remaining
		entries zero; then $DX(p)=J$ and $D(DX)_{p}[q_{0}]=L$.
	\end{proof}
	
	It remains to decide for which $\rho$ the last term of \eqref{eq:allcoeff}
	is a nonzero linear form in $\Ltr$, that is, for which $j$ the
	coefficient $A_{j}$ of the adjugate is a nonzero matrix.  Two of them
	are immediate: $A_{\nu-1}=I$, and $A_{0}=\adj(-\Jtr)$ is invertible
	because $\Jtr$ is.  The remaining ones turn out to follow from an
	elementary statement about gaps in the coefficient sequence of a
	polynomial, which we prove for want of a reference.
	
	\begin{lemma}[Gap lemma]\label{lem:gap}
		Let $K$ be a field of characteristic zero and let
		$v=\sum_{j=0}^{m}a_{j}\lambda^{j}\in K[\lambda]$ have $a_{m}\neq0$
		and $a_{0}\neq0$.  Let $\varepsilon$ be the number of distinct roots
		of $v$ in an algebraic closure of $K$.  If
		\[
		a_{r}=a_{r+1}=\dots=a_{r+g-1}=0
		\qquad\text{for some } r\ge1 \text{ with } r+g-1\le m-1,
		\]
		then $g\le\varepsilon-1$.  Equivalently, a gap of $g$ consecutive
		vanishing coefficients strictly between the two extreme ones forces
		at least $g+1$ distinct roots.  The bound is attained by
		$\lambda^{m}-\gamma$, $\gamma\neq0$.
	\end{lemma}
	
	\begin{proof}
		We may work over an algebraic closure.  Write
		$v=\gamma\prod_{i=1}^{\varepsilon}(\lambda-\alpha_{i})^{\kappa_{i}}$
		with the $\alpha_{i}$ distinct, $\sum_{i}\kappa_{i}=m$, and
		$\alpha_{i}\neq0$ for
		all $i$ because $a_{0}\neq0$.  Put
		\[
		\pi:=\prod_{i=1}^{\varepsilon}(\lambda-\alpha_{i})
		=\sum_{i=0}^{\varepsilon}\sigma_{i}\lambda^{\varepsilon-i},
		\qquad
		N:=\sum_{i=1}^{\varepsilon}\kappa_{i}\,\frac{\pi}{\lambda-\alpha_{i}}
		=\sum_{i=0}^{\varepsilon-1}N_{i}\lambda^{\varepsilon-1-i},
		\]
		so that $\sigma_{0}=1$ and $N_{0}=\sum_{i}\kappa_{i}=m$.  From
		$v'/v=\sum_{i}\kappa_{i}/(\lambda-\alpha_{i})$ we obtain the polynomial
		identity $\pi\,v'=N\,v$, and comparing the coefficients of
		$\lambda^{k+\varepsilon-1}$ on the two sides gives, for every
		$k\in\mathbb Z$ and with the convention $a_{j}:=0$ for $j<0$ or
		$j>m$,
		\begin{equation}\label{eq:recur}
			\sigma_{\varepsilon}(k+\varepsilon)\,a_{k+\varepsilon}
			+\sum_{i=0}^{\varepsilon-1}
			\bigl(\sigma_{i}(k+i)-N_{i}\bigr)a_{k+i}=0 .
		\end{equation}
		Suppose now $a_{r}=\dots=a_{r+g-1}=0$ as in the statement, and assume
		$g\ge\varepsilon$; in particular $a_{r}=\dots=a_{r+\varepsilon-1}=0$.
		Apply \eqref{eq:recur} with $k=r-1$.  The first term involves
		$a_{r+\varepsilon-1}=0$; the terms $i=1,\dots,\varepsilon-1$ involve
		$a_{r},\dots,a_{r+\varepsilon-2}$, all zero; and the term $i=0$ is
		$\bigl(\sigma_{0}(r-1)-N_{0}\bigr)a_{r-1}=(r-1-m)\,a_{r-1}$.  Since
		$r-1<m$ and $\operatorname{char}K=0$, we conclude $a_{r-1}=0$.  Thus
		$a_{r-1},\dots,a_{r+\varepsilon-2}$ is again a block of $\varepsilon$
		consecutive zeros, one step lower.  Iterating $r$ times yields
		$a_{0}=0$, contradicting $a_{0}\neq0$.  Hence $g\le\varepsilon-1$.
		For $v=\lambda^{m}-\gamma$ one has $\varepsilon=m$ and $g=m-1$.
	\end{proof}
	
	\begin{theorem}[Sharpness]\label{thm:sharp}
		Let $A$ be an invertible $\nu\times\nu$ matrix over a field of
		characteristic zero and write
		$\adj(\lambda I-A)=\sum_{j=0}^{\nu-1}A_{j}\lambda^{j}$.  Then
		$A_{j}\neq0$ for every $0\le j\le\nu-1$.  Consequently, under
		$(\mathrm{BT}_{0})$ with $n\ge3$, the map $\Ltr\mapsto t_{\rho-2}$ is
		a nonzero linear form for every $2\le\rho\le n-1$, and therefore
		\emph{no} $D_{q_{0}}e_{k}$ with $k\le n-2$ is determined by $J$ and
		the pair $(a,b)$: the two invariants $e_{n-1}$ and $e_{n}$ of
		Theorem~\ref{thm:main} are the only transversally clean ones.
	\end{theorem}
	
	\begin{proof}
		Write $\adj(\lambda I-A)=\sum_{k=0}^{\nu-1}B_{k}\lambda^{\nu-1-k}$,
		so $A_{j}=B_{\nu-1-j}$.  Expanding
		$\adj(I-\zeta A)=\dt(I-\zeta A)(I-\zeta A)^{-1}$ in powers of
		$\zeta$ gives
		$B_{k}=P_{k}(A)$, where $P_{k}$ is the truncation of
		$\chi_{A}$ formed by its $k+1$ leading coefficients; thus $P_{k}$ is
		monic of degree $k$, and $\chi_{A}=P_{k}\lambda^{\nu-k}+R_{k}$ with
		$\deg R_{k}\le\nu-k-1$ and
		$R_{k}(0)=\chi_{A}(0)=(-1)^{\nu}\dt A\neq0$.
		
		Assume $B_{k}=0$ for some $k$, and let $\pi$ be the minimal
		polynomial of $A$, of degree $D\ge1$.  Then $\pi\mid P_{k}$, so
		$D\le k$; and $R_{k}(A)=\chi_{A}(A)-P_{k}(A)A^{\nu-k}=0$ with
		$R_{k}\neq0$, so $\pi\mid R_{k}$ and $D\le\nu-k-1$.  Write
		$P_{k}=\pi\,u$, $R_{k}=\pi\,w$ and $\chi_{A}=\pi\,v$.  Dividing
		$\chi_{A}=P_{k}\lambda^{\nu-k}+R_{k}$ by $\pi$ gives
		\[
		v=u\,\lambda^{\nu-k}+w,\qquad \deg w\le\nu-k-1-D,
		\]
		so the coefficients of $v$ in the degrees
		$\nu-k-D,\dots,\nu-k-1$ all vanish.  This is a gap of length $D$
		lying strictly inside the coefficient sequence of $v$: its lowest
		degree is $\nu-k-D\ge1$ because $D\le\nu-k-1$, and its highest is
		$\nu-k-1\le\deg v-1=\nu-D-1$ because $D\le k$.  Moreover
		$v(0)=\chi_{A}(0)/\pi(0)\neq0$, and $\deg v=\nu-D\ge k+1\ge D+1\ge2$.
		Finally the number $\varepsilon$ of distinct roots of $v$ is at most
		the number of distinct eigenvalues of $A$, which is at most
		$D=\deg\pi$.  The gap lemma now gives $D\le\varepsilon-1\le D-1$, a
		contradiction.  Hence $B_{k}\neq0$ for every $k$.
		
		For the consequence, apply this with $A=\Jtr$, which is invertible by
		Lemma~\ref{lem:split}: every $A_{j}$ is a nonzero matrix, and the
		trace pairing being nondegenerate, $\Ltr\mapsto\tr(A_{j}\Ltr)$ is a
		nonzero linear form.  Since $\Ltr$ may be prescribed arbitrarily by
		Proposition~\ref{prop:sharp}, the last term of \eqref{eq:allcoeff}
		varies independently of $J$ and $(a,b)$ whenever $\rho\ge2$.
	\end{proof}
	
	\begin{remark}
		Lemma~\ref{lem:gap} is a companion to the classical bound of Haj\'os,
		which states that over a field of characteristic zero the number of
		nonzero coefficients of a polynomial exceeds the multiplicity of any
		of its nonzero roots.  The two statements exchange the roles of the
		sparsity pattern and of the root structure, and neither seems to
		follow formally from the other.  For $\varepsilon=1$,
		Lemma~\ref{lem:gap} recovers the fact that
		$P_{k}(\alpha)=(-1)^{k}\binom{\nu-1}{k}\alpha^{k}\neq0$ when
		$A=\alpha I+N$ with $N$ nilpotent and $\alpha\neq0$, which is the
		case where the theorem can also be checked in closed form.
	\end{remark}
	
	Identity \eqref{eq:trace} is the exact form of what is often described
	informally as hyperbolic contamination: the trace of the full Jacobian
	does vary along $\ker J$ at the rate $b$, but with an additive error
	equal to the variation of the trace of the transverse block.  It is the
	vanishing of the two lowest coefficients of $\chi_{J}$, equivalently the
	factor $\lambda^{2}$ multiplying the transverse term in \eqref{eq:gen},
	that makes $e_{n-1}$ and $e_{n}$, and only those, clean.
	
	\subsection{Relation with the spectral-factorization argument}
	
	If the spectrum of $J$ splits into $\{0,0\}$ and the transverse part, one
	may factor
	$\chi_{DX(x)}(\lambda)=P_{c}(\lambda,x)P_{\mathrm{tr}}(\lambda,x)$ with
	$P_{c}=\lambda^{2}-\tau_{c}(x)\lambda+\delta_{c}(x)$ near $p$, obtaining
	$e_{n}=\delta_{c}e^{\mathrm{tr}}_{n-2}$ and
	$e_{n-1}=\tau_{c}e^{\mathrm{tr}}_{n-2}+\delta_{c}e^{\mathrm{tr}}_{n-3}$;
	differentiating along $q_{0}$ and using
	$\tau_{c}(p)=\delta_{c}(p)=0$ yields \eqref{eq:mainbis} once one knows
	$a=-\tfrac12D_{q_{0}}\delta_{c}$ and $b=D_{q_{0}}\tau_{c}$.  This route is
	correct, but it costs three extra ingredients: the smoothness of the
	factorization; the identification of $P_{c}$ restricted to a centre
	manifold with the characteristic polynomial of the reduced Jacobian; and
	the identification of $D_{q_{0}}\tau_{c}$, $D_{q_{0}}\delta_{c}$ with
	$b$, $-2a$, which needs a normal-form transformation and hence
	Lemma~\ref{lem:inv}.  Each imports hypotheses that
	Theorem~\ref{thm:main} does not need.
	
	\begin{remark}[Beyond hyperbolicity]\label{rem:strength}
		Theorem~\ref{thm:main} holds verbatim when $\Jtr$ has purely imaginary
		eigenvalues, or eigenvalues with $\mu_{i}+\mu_{j}=0$, situations in
		which no invariant two-dimensional manifold tangent to $\Ec$ need
		exist and the perturbative argument breaks down.  In such cases
		\eqref{eq:mainbis} is still an identity between the $2$-jet quantities
		\eqref{eq:abKuz} and the spectral data, and \eqref{eq:abKuz} still
		computes the quadratic part of $\Pi_{c}X$ restricted to $\Ec$; what is
		lost is only the reading of $a,b$ as coefficients of a two-dimensional
		germ, i.e.\ Lemma~\ref{lem:identify}.  See Example~\ref{ex:nonhyp}.
	\end{remark}
	
	\section{Geometric reading of the two identities}\label{sec:geom}
	
	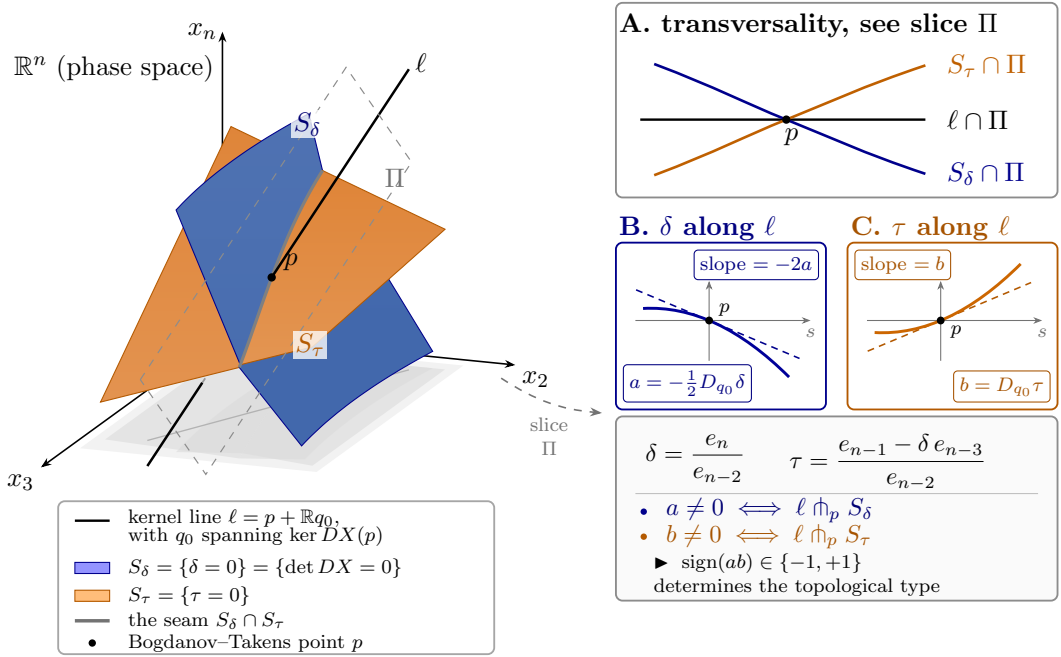
\begin{figure}[!t]
		\centering
		\begin{tikzpicture}[
			scale=0.92, transform shape,
			font=\footnotesize,
			panel/.style={draw=#1, rounded corners=2.5pt, line width=0.7pt,
				fill=white},
			ptitle/.style={font=\small\bfseries, anchor=north west, inner sep=0pt},
			tag/.style={draw=#1, text=#1, fill=white, rounded corners=1.5pt,
				line width=0.4pt, inner sep=2.5pt, font=\scriptsize}
			]
			
			\begin{scope}
				\fill[black!5] (-0.646,-0.927) -- (-2.756,-2.458) -- (0.646,-2.799) -- (2.756,-1.268) -- (2.756,-1.268) -- cycle;
				\fill[black!14,fill opacity=0.6] (2.315,-1.307) -- (2.149,-1.397) -- (1.988,-1.488) -- (1.832,-1.579) -- (1.682,-1.671) -- (1.537,-1.763) -- (1.398,-1.856) -- (1.264,-1.949) -- (1.135,-2.043) -- (1.012,-2.138) -- (0.894,-2.233) -- (0.781,-2.328) -- (0.674,-2.425) -- (0.572,-2.521) -- (0.476,-2.618) -- (-0.466,-2.604) -- (-0.408,-2.503) -- (-0.350,-2.402) -- (-0.290,-2.301) -- (-0.229,-2.201) -- (-0.165,-2.100) -- (-0.098,-2.000) -- (-0.028,-1.900) -- (0.047,-1.801) -- (0.128,-1.702) -- (0.214,-1.604) -- (0.307,-1.506) -- (0.408,-1.410) -- (0.518,-1.314) -- (0.636,-1.219) -- (0.731,-1.148) -- cycle;
				\fill[black!14,fill opacity=0.6] (0.731,-1.148) -- (0.606,-1.242) -- (0.489,-1.338) -- (0.382,-1.434) -- (0.283,-1.531) -- (0.192,-1.628) -- (0.107,-1.727) -- (0.028,-1.826) -- (-0.046,-1.925) -- (-0.115,-2.025) -- (-0.181,-2.125) -- (-0.244,-2.226) -- (-0.305,-2.327) -- (-0.365,-2.428) -- (-0.423,-2.529) -- (-1.375,-2.513) -- (-1.283,-2.416) -- (-1.185,-2.319) -- (-1.082,-2.222) -- (-0.973,-2.126) -- (-0.859,-2.031) -- (-0.740,-1.936) -- (-0.615,-1.842) -- (-0.485,-1.748) -- (-0.350,-1.654) -- (-0.209,-1.562) -- (-0.063,-1.469) -- (0.089,-1.378) -- (0.246,-1.287) -- (0.408,-1.196) -- (0.533,-1.128) -- cycle;
				\fill[black!14,fill opacity=0.6] (-0.585,-1.017) -- (-0.714,-1.110) -- (-0.844,-1.204) -- (-0.973,-1.298) -- (-1.102,-1.392) -- (-1.232,-1.486) -- (-1.361,-1.580) -- (-1.490,-1.674) -- (-1.620,-1.768) -- (-1.749,-1.862) -- (-1.879,-1.956) -- (-2.008,-2.050) -- (-2.137,-2.143) -- (-2.267,-2.237) -- (-2.396,-2.331) -- (-0.466,-2.604) -- (-0.408,-2.503) -- (-0.350,-2.402) -- (-0.290,-2.301) -- (-0.229,-2.201) -- (-0.165,-2.100) -- (-0.098,-2.000) -- (-0.028,-1.900) -- (0.047,-1.801) -- (0.128,-1.702) -- (0.214,-1.604) -- (0.307,-1.506) -- (0.408,-1.410) -- (0.518,-1.314) -- (0.636,-1.219) -- (0.731,-1.148) -- cycle;
				\fill[black!14,fill opacity=0.6] (0.731,-1.148) -- (0.606,-1.242) -- (0.489,-1.338) -- (0.382,-1.434) -- (0.283,-1.531) -- (0.192,-1.628) -- (0.107,-1.727) -- (0.028,-1.826) -- (-0.046,-1.925) -- (-0.115,-2.025) -- (-0.181,-2.125) -- (-0.244,-2.226) -- (-0.305,-2.327) -- (-0.365,-2.428) -- (-0.423,-2.529) -- (0.585,-2.709) -- (0.714,-2.616) -- (0.844,-2.522) -- (0.973,-2.428) -- (1.102,-2.334) -- (1.232,-2.240) -- (1.361,-2.146) -- (1.490,-2.052) -- (1.620,-1.958) -- (1.749,-1.864) -- (1.879,-1.770) -- (2.008,-1.676) -- (2.137,-1.583) -- (2.267,-1.489) -- (2.396,-1.395) -- (2.493,-1.324) -- cycle;
				\draw[black!30,line width=0.5pt] (-1.788,-2.346) -- (-1.156,-2.175) -- (-0.524,-2.004) -- (0.108,-1.834) -- (0.740,-1.663) -- (1.372,-1.492) -- (1.973,-1.330);
				\draw[black!25,dashed,line width=0.3pt] (0.000,0.000) -- (0.000,-1.863);
				\draw[-{Stealth[length=4pt]},line width=.6pt] (-0.708,-0.838) -- (3.504,-1.260) node[below right=-2pt,font=\small] {$x_{2}$};
				\draw[-{Stealth[length=4pt]},line width=.6pt] (-0.708,-0.838) -- (-3.319,-2.734) node[below left=-2pt,font=\small] {$x_{3}$};
				\draw[-{Stealth[length=4pt]},line width=.6pt] (-0.708,-0.838) -- (-0.708,3.536) node[left=-2pt,font=\small] {$x_{n}$};
				\node[below left=-3pt,font=\small] at (-0.708,-0.838) {$0$};
				\draw[black!40,line width=.7pt,dash pattern=on 2pt off 2pt] (-0.998,-1.515) -- (-0.745,-1.131) -- (-0.492,-0.747) -- (-0.239,-0.363) -- (0.013,0.020) -- (0.013,0.020);
				\definecolor{tbt}{RGB}{223,133,57}
				\definecolor{tbb}{RGB}{227,147,80}
				\shade[top color=tbt,bottom color=tbb,draw=orange!70!black,line width=.4pt] (-0.585,2.170) -- (-0.650,2.033) -- (-0.714,1.897) -- (-0.779,1.760) -- (-0.844,1.623) -- (-0.908,1.487) -- (-0.973,1.350) -- (-1.038,1.213) -- (-1.102,1.077) -- (-1.167,0.940) -- (-1.232,0.803) -- (-1.296,0.666) -- (-1.361,0.530) -- (-1.426,0.393) -- (-1.490,0.256) -- (-1.555,0.120) -- (-1.620,-0.017) -- (-1.685,-0.154) -- (-1.749,-0.290) -- (-1.814,-0.427) -- (-1.879,-0.564) -- (-1.943,-0.700) -- (-2.008,-0.837) -- (-2.073,-0.974) -- (-2.137,-1.110) -- (-2.202,-1.247) -- (-2.267,-1.384) -- (-2.331,-1.520) -- (-2.396,-1.657) -- (-2.461,-1.794) -- (-0.466,-1.295) -- (-0.437,-1.219) -- (-0.408,-1.142) -- (-0.379,-1.063) -- (-0.350,-0.982) -- (-0.320,-0.899) -- (-0.290,-0.814) -- (-0.260,-0.728) -- (-0.229,-0.640) -- (-0.197,-0.550) -- (-0.165,-0.458) -- (-0.132,-0.364) -- (-0.098,-0.269) -- (-0.063,-0.173) -- (-0.028,-0.074) -- (0.009,0.025) -- (0.047,0.126) -- (0.087,0.228) -- (0.128,0.331) -- (0.170,0.434) -- (0.214,0.539) -- (0.260,0.644) -- (0.307,0.750) -- (0.357,0.856) -- (0.408,0.962) -- (0.462,1.068) -- (0.518,1.174) -- (0.576,1.279) -- (0.636,1.383) -- (0.699,1.487) -- (0.731,1.538) -- cycle;
				\definecolor{dbt}{RGB}{67,105,173}
				\definecolor{dbb}{RGB}{66,104,172}
				\shade[top color=dbt,bottom color=dbb,draw=blue!55!black,line width=.4pt] (2.315,-1.064) -- (2.232,-1.109) -- (2.149,-1.154) -- (2.068,-1.199) -- (1.988,-1.245) -- (1.910,-1.290) -- (1.832,-1.336) -- (1.757,-1.382) -- (1.682,-1.428) -- (1.609,-1.474) -- (1.537,-1.520) -- (1.467,-1.566) -- (1.398,-1.613) -- (1.330,-1.660) -- (1.264,-1.706) -- (1.199,-1.753) -- (1.135,-1.800) -- (1.073,-1.847) -- (1.012,-1.895) -- (0.952,-1.942) -- (0.894,-1.990) -- (0.837,-2.038) -- (0.781,-2.085) -- (0.727,-2.133) -- (0.674,-2.182) -- (0.623,-2.230) -- (0.572,-2.278) -- (0.523,-2.327) -- (0.476,-2.375) -- (0.430,-2.424) -- (-0.466,-1.295) -- (-0.437,-1.219) -- (-0.408,-1.142) -- (-0.379,-1.063) -- (-0.350,-0.982) -- (-0.320,-0.899) -- (-0.290,-0.814) -- (-0.260,-0.728) -- (-0.229,-0.640) -- (-0.197,-0.550) -- (-0.165,-0.458) -- (-0.132,-0.364) -- (-0.098,-0.269) -- (-0.063,-0.173) -- (-0.028,-0.074) -- (0.009,0.025) -- (0.047,0.126) -- (0.087,0.228) -- (0.128,0.331) -- (0.170,0.434) -- (0.214,0.539) -- (0.260,0.644) -- (0.307,0.750) -- (0.357,0.856) -- (0.408,0.962) -- (0.462,1.068) -- (0.518,1.174) -- (0.576,1.279) -- (0.636,1.383) -- (0.699,1.487) -- (0.731,1.538) -- cycle;
				\definecolor{dat}{RGB}{79,114,178}
				\definecolor{dab}{RGB}{72,109,175}
				\shade[top color=dat,bottom color=dab,draw=blue!55!black,line width=.4pt] (0.731,1.538) -- (0.667,1.435) -- (0.606,1.331) -- (0.546,1.226) -- (0.489,1.121) -- (0.435,1.015) -- (0.382,0.909) -- (0.332,0.803) -- (0.283,0.697) -- (0.236,0.592) -- (0.192,0.487) -- (0.149,0.382) -- (0.107,0.279) -- (0.067,0.177) -- (0.028,0.075) -- (-0.009,-0.025) -- (-0.046,-0.124) -- (-0.081,-0.221) -- (-0.115,-0.317) -- (-0.148,-0.411) -- (-0.181,-0.504) -- (-0.213,-0.595) -- (-0.244,-0.684) -- (-0.275,-0.771) -- (-0.305,-0.857) -- (-0.335,-0.941) -- (-0.365,-1.023) -- (-0.394,-1.103) -- (-0.423,-1.181) -- (-0.452,-1.257) -- (-1.375,0.970) -- (-1.329,1.018) -- (-1.283,1.067) -- (-1.234,1.116) -- (-1.185,1.164) -- (-1.134,1.213) -- (-1.082,1.261) -- (-1.028,1.309) -- (-0.973,1.357) -- (-0.917,1.405) -- (-0.859,1.452) -- (-0.800,1.500) -- (-0.740,1.547) -- (-0.678,1.594) -- (-0.615,1.641) -- (-0.551,1.688) -- (-0.485,1.735) -- (-0.418,1.782) -- (-0.350,1.829) -- (-0.280,1.875) -- (-0.209,1.921) -- (-0.136,1.968) -- (-0.063,2.014) -- (0.012,2.059) -- (0.089,2.105) -- (0.167,2.151) -- (0.246,2.196) -- (0.326,2.242) -- (0.408,2.287) -- (0.491,2.332) -- (0.533,2.355) -- cycle;
				\definecolor{tat}{RGB}{223,132,55}
				\definecolor{tab}{RGB}{227,149,82}
				\shade[top color=tat,bottom color=tab,draw=orange!70!black,line width=.4pt] (0.731,1.538) -- (0.667,1.435) -- (0.606,1.331) -- (0.546,1.226) -- (0.489,1.121) -- (0.435,1.015) -- (0.382,0.909) -- (0.332,0.803) -- (0.283,0.697) -- (0.236,0.592) -- (0.192,0.487) -- (0.149,0.382) -- (0.107,0.279) -- (0.067,0.177) -- (0.028,0.075) -- (-0.009,-0.025) -- (-0.046,-0.124) -- (-0.081,-0.221) -- (-0.115,-0.317) -- (-0.148,-0.411) -- (-0.181,-0.504) -- (-0.213,-0.595) -- (-0.244,-0.684) -- (-0.275,-0.771) -- (-0.305,-0.857) -- (-0.335,-0.941) -- (-0.365,-1.023) -- (-0.394,-1.103) -- (-0.423,-1.181) -- (-0.452,-1.257) -- (0.585,-1.000) -- (0.650,-0.943) -- (0.714,-0.886) -- (0.779,-0.828) -- (0.844,-0.771) -- (0.908,-0.713) -- (0.973,-0.656) -- (1.038,-0.599) -- (1.102,-0.541) -- (1.167,-0.484) -- (1.232,-0.426) -- (1.296,-0.369) -- (1.361,-0.312) -- (1.426,-0.254) -- (1.490,-0.197) -- (1.555,-0.140) -- (1.620,-0.082) -- (1.685,-0.025) -- (1.749,0.033) -- (1.814,0.090) -- (1.879,0.147) -- (1.943,0.205) -- (2.008,0.262) -- (2.073,0.320) -- (2.137,0.377) -- (2.202,0.434) -- (2.267,0.492) -- (2.331,0.549) -- (2.396,0.606) -- (2.461,0.664) -- (2.493,0.693) -- cycle;
				\draw[black!50,line width=1.2pt] (0.731,1.538) -- (0.636,1.383) -- (0.546,1.226) -- (0.462,1.068) -- (0.382,0.909) -- (0.307,0.750) -- (0.236,0.592) -- (0.170,0.434) -- (0.107,0.279) -- (0.047,0.126) -- (-0.009,-0.025) -- (-0.063,-0.173) -- (-0.115,-0.317) -- (-0.165,-0.458) -- (-0.213,-0.595) -- (-0.260,-0.728) -- (-0.305,-0.857) -- (-0.350,-0.982) -- (-0.394,-1.103) -- (-0.437,-1.219) -- (-0.466,-1.295);
				\draw[black,line width=.95pt] (-1.788,-2.714) -- (-1.535,-2.330) -- (-1.282,-1.947) -- (-1.029,-1.563) -- (-0.998,-1.515);
				\draw[black,line width=.95pt] (0.013,0.020) -- (0.266,0.404) -- (0.519,0.788) -- (0.772,1.172) -- (1.025,1.556) -- (1.277,1.940) -- (1.530,2.323) -- (1.783,2.707) -- (1.973,2.995);
				\node[font=\small] at (2.15,3.087) {$\ell$};
				\draw[black!45,dashed,line width=.45pt] (-1.889,-1.471) -- (1.070,3.022) -- (2.012,1.658) -- (-0.947,-2.834) -- cycle;
				\node[black!55,font=\small,below left=-2pt] at (2.012,1.658) {$\Pi$};
				\fill[black] (0.000,0.000) circle (1.7pt);
				\node[font=\small,above right=-1pt and 1pt] at (0.000,0.000) {$p$};
				\node[font=\small,color=blue!45!black,fill=white,fill opacity=.78,text opacity=1,inner sep=1pt] at (0.516,2.202) {$S_{\delta}$};
				\node[font=\small,color=orange!70!black,fill=white,fill opacity=.78,text opacity=1,inner sep=1pt] at (0.554,-0.977) {$S_{\tau}$};
				\node[font=\small, anchor=north west] at (-3.85,3.35)
				{$\R^{n}$ (phase space)};
			\end{scope}
			
			\begin{scope}[shift={(-2.95,-5.30)}, font=\scriptsize]
				\draw[panel=black!35] (-0.12,-0.12) rectangle (6.15,2.08);
				\draw[black, line width=0.9pt] (0.12,1.80) -- (0.62,1.80);
				\node[anchor=west, align=left] at (0.74,1.70)
				{kernel line $\ell=p+\R q_{0}$,\\[-2pt]
					with $q_{0}$ spanning $\ker DX(p)$};
				\fill[blue!42!white] (0.12,1.04) rectangle (0.62,1.20);
				\draw[blue!55!black, line width=0.4pt]
				(0.12,1.04) rectangle (0.62,1.20);
				\node[anchor=west] at (0.74,1.12)
				{$S_{\delta}=\{\delta=0\}=\{\dt DX=0\}$};
				\fill[orange!52!white] (0.12,0.64) rectangle (0.62,0.80);
				\draw[orange!70!black, line width=0.4pt]
				(0.12,0.64) rectangle (0.62,0.80);
				\node[anchor=west] at (0.74,0.72) {$S_{\tau}=\{\tau=0\}$};
				\draw[black!55, line width=1.3pt] (0.12,0.36) -- (0.62,0.36);
				\node[anchor=west] at (0.74,0.36)
				{the seam $S_{\delta}\cap S_{\tau}$};
				\fill[black] (0.37,0.04) circle (1.5pt);
				\node[anchor=west] at (0.74,0.04)
				{Bogdanov--Takens point $p$};
			\end{scope}
			
			\draw[-{Stealth[length=4.5pt]}, black!55, dashed, line width=0.6pt]
			(3.30,-1.45) to[bend right=14] node[midway, below=2pt, align=center,
			font=\scriptsize, text=black!55] {slice\\$\Pi$} (4.75,-1.95);
			
			\begin{scope}[shift={(4.95,0.35)}]
				
				\begin{scope}
					\draw[panel=black!45] (0,0.80) rectangle (6.35,3.6);
					\node[ptitle] at (0.04,3.46)
					{A.\ transversality, see slice $\Pi$};
					\draw[black, line width=0.9pt] (0.35,1.92) -- (4.45,1.92);
					\node[anchor=west, font=\small] at (4.62,1.92) {$\ell\cap\Pi$};
					\draw[blue!55!black, line width=1pt]
					(0.55,2.72) .. controls (1.45,2.38) and (1.85,2.16)
					.. (2.45,1.92) .. controls (3.05,1.68) and (3.55,1.44)
					.. (4.45,1.14);
					\draw[orange!75!black, line width=1pt]
					(0.55,1.12) .. controls (1.45,1.46) and (1.85,1.68)
					.. (2.45,1.92) .. controls (3.05,2.16) and (3.55,2.40)
					.. (4.45,2.70);
					\node[blue!55!black, font=\small, anchor=west]
					at (4.62,1.14) {$S_{\delta}\cap\Pi$};
					\node[orange!75!black, font=\small, anchor=west]
					at (4.62,2.70) {$S_{\tau}\cap\Pi$};
					\fill[black] (2.45,1.92) circle (1.7pt);
					\node[font=\small, below right=-1pt and 1pt] at (2.25,1.90) {$p$};
				\end{scope}
				
				\begin{scope}
					\draw[panel=blue!55!black] (0,-2.25) rectangle (3.02,0.15);
					\node[ptitle, blue!50!black] at (0.04,0.54)
					{B.\ $\delta$ along $\ell$};
					\draw[-{Stealth[length=3.4pt]}, black!55, line width=0.4pt]
					(0.28,-0.97) -- (2.80,-0.97)
					node[below=-1pt, font=\scriptsize] {$s$};
					\draw[-{Stealth[length=3.4pt]}, black!55, line width=0.4pt]
					(1.34,-1.55) -- (1.34,-0.40);
					\draw[blue!55!black, dash pattern=on 2.6pt off 2pt,
					line width=0.6pt] (0.29,-0.529) -- (2.64,-1.516);
					\draw[blue!60!black, line width=1.1pt]
					(0.39,-0.797) .. controls (1.09,-0.758) and (1.79,-1.087)
					.. (2.49,-1.784);
					\fill[black] (1.34,-0.97) circle (1.6pt);
					\node[font=\scriptsize, above right=-2pt and 0pt]
					at (1.34,-0.97) {$p$};
					\node[tag=blue!50!black, anchor=north east]
					at (2.92,0.05) {slope $=-2a$};
					\node[tag=blue!50!black, anchor=south west]
					at (0.10,-2.15) {$a=-\tfrac12 D_{q_{0}}\delta$};
				\end{scope}
				
				\begin{scope}[shift={(3.33,0)}]
					\draw[panel=orange!70!black] (0,-2.25) rectangle (3.02,0.15);
					\node[ptitle, orange!65!black] at (0.04,0.54)
					{C.\ $\tau$ along $\ell$};
					\draw[-{Stealth[length=3.4pt]}, black!55, line width=0.4pt]
					(0.28,-0.97) -- (2.80,-0.97)
					node[below=-1pt, font=\scriptsize] {$s$};
					\draw[-{Stealth[length=3.4pt]}, black!55, line width=0.4pt]
					(1.34,-1.55) -- (1.34,-0.40);
					\draw[orange!70!black, dash pattern=on 2.6pt off 2pt,
					line width=0.6pt] (0.29,-1.411) -- (2.64,-0.424);
					\draw[orange!80!black, line width=1.1pt]
					(0.39,-1.143) .. controls (1.09,-1.182) and (1.79,-0.853)
					.. (2.49,-0.156);
					\fill[black] (1.34,-0.97) circle (1.6pt);
					\node[font=\scriptsize, below right=-2pt and 0pt]
					at (1.34,-0.97) {$p$};
					\node[tag=orange!65!black, anchor=north west]
					at (0.10, 0.05) {slope $=b$};
					\node[tag=orange!65!black, anchor=south east]
					at (2.92,-2.15) {$b=D_{q_{0}}\tau$};
				\end{scope}
				
				\begin{scope}
					\draw[panel=black!45, fill=black!2]
					(0,-5) rectangle (6.35,-2.35);
					\node[anchor=north west, font=\footnotesize] at (0.28,-2.46)
					{$\displaystyle \delta=\frac{e_{n}}{e_{n-2}}$};
					\node[anchor=north west, font=\footnotesize] at (2.35,-2.46)
					{$\displaystyle \tau=\frac{e_{n-1}-\delta\,e_{n-3}}{e_{n-2}}$};
					\draw[black!25, line width=0.4pt] (0.28,-3.46) -- (6.07,-3.46);
					\fill[blue!55!black] (0.40,-3.72) circle (1.3pt);
					\node[anchor=west, font=\footnotesize, text=blue!45!black]
					at (0.58,-3.7) {$a\neq0 \iff \ell\pitchfork_{p}S_{\delta}$};
					\fill[orange!70!black] (0.40,-4.08) circle (1.3pt);
					\node[anchor=west, font=\footnotesize, text=orange!65!black]
					at (0.58,-4.08) {$b\neq0 \iff \ell\pitchfork_{p}S_{\tau}$};
					\node[anchor=west, font=\scriptsize] at (0.40,-4.46)
					{$\blacktriangleright\ \operatorname{sign}(ab)\in\{-1,+1\}$
					};
					\node[anchor=west,font=\scriptsize]
					at (0.40,-4.8)
					{determines the topological type};
				\end{scope}
				
			\end{scope}
			
		\end{tikzpicture}
		
		\caption{\label{fig:transv}%
			\textbf{Geometric reading of the two identities.}
			The central determinant $\delta=e_{n}/e_{n-2}$ and the central
			trace $\tau=(e_{n-1}-\delta\,e_{n-3})/e_{n-2}$ of
			Definition~\ref{def:deltatau} are rational functions of the
			principal minors of $DX$: no eigenvector of $DX$ or of its
			adjoint, no generalized eigenvector and no local centre manifold
			enters their construction.  Both vanish at the Bogdanov--Takens
			point $p$, so their zero sets $S_{\delta}$ and $S_{\tau}$ pass
			through $p$ and meet along the codimension-two seam drawn in
			grey.  By Corollary~\ref{cor:transv} the nondegeneracy conditions
			$a\neq0$ and $b\neq0$ are exactly transversality of the kernel
			line $\ell=p+\R q_{0}$ to $S_{\delta}$ and to $S_{\tau}$, visible
			in the slice $\Pi$ as two nonzero crossing angles (panel A);
			and by Proposition~\ref{prop:geom} the two derivatives that
			measure this transversality are the BT coefficients themselves,
			since $\delta$ and $\tau$ restricted to $\ell$ vanish at $p$ with
			slopes $-2a$ and $b$; in panels B and C the solid curve is the restriction and the dashed line its tangent at $p$.  Two cautions.  The
			three-dimensional drawing is a schematic: for $n\ge3$ the sets
			$S_{\delta}$ and $S_{\tau}$ are hypersurfaces of $\R^{n}$, and the
			sheets shown are two-dimensional slices of them.  And $a\neq0$
			together with $b\neq0$ gives transversality of $\ell$ to each
			hypersurface separately; it does not by itself make $S_{\delta}$
			and $S_{\tau}$ transverse to each other.}
	\end{figure}
	
	The identities \eqref{eq:mainbis} become geometric once one names the two
	scalar functions they are the derivatives of.
	
	\begin{definition}\label{def:deltatau}
		Let $U'\subset U$ be the open set where $e_{n-2}(DX(x))\neq0$.  On
		$U'$ define the \emph{central determinant} and the \emph{central
			trace} of $X$ by
		\begin{equation}\label{eq:deltatau}
			\delta:=\frac{e_{n}(DX)}{e_{n-2}(DX)},
			\qquad
			\tau:=\frac{e_{n-1}(DX)-\delta\,e_{n-3}(DX)}{e_{n-2}(DX)},
		\end{equation}
		equivalently
		$\tau=\bigl(e_{n-1}(DX)e_{n-2}(DX)-e_{n}(DX)e_{n-3}(DX)\bigr)
		/e_{n-2}(DX)^{2}$.
		Both are rational functions of the principal minors of $DX$, and their
		evaluation requires no eigenvector of $DX$ or of its adjoint, no
		explicit local centre manifold, and no numerical spectral
		factorization: one characteristic polynomial delivers all four
		invariants involved.  For $n=2$ they are $\delta=\dt DX$ and
		$\tau=\tr DX$.
	\end{definition}
	
	\begin{proposition}\label{prop:geom}
		Under $(\mathrm{BT}_{0})$ with $X\in C^{2}$,
		\[
		\delta(p)=\tau(p)=0,
		\qquad
		a=-\tfrac12\,D_{q_{0}}\delta,
		\qquad
		b=D_{q_{0}}\tau .
		\]
	\end{proposition}
	
	\begin{proof}
		$\delta(p)=0$ and $\tau(p)=0$ because $e_{n}(J)=e_{n-1}(J)=0$.  By the
		quotient rule and $e_{n}(J)=0$,
		$D_{q_{0}}\delta=D_{q_{0}}e_{n}/e_{n-2}=-2a$ by \eqref{eq:two}.  For
		$\tau$, the numerator $\mathcal N:=e_{n-1}e_{n-2}-e_{n}e_{n-3}$ vanishes
		at $p$, so $D_{q_{0}}\tau=D_{q_{0}}\mathcal N/e_{n-2}^{2}$, and using
		$e_{n-1}(J)=e_{n}(J)=0$ again,
		\[
		D_{q_{0}}\mathcal N=e_{n-2}D_{q_{0}}e_{n-1}-e_{n-3}D_{q_{0}}e_{n}
		=e_{n-2}\bigl(be_{n-2}-2ae_{n-3}\bigr)-e_{n-3}\bigl(-2ae_{n-2}\bigr)
		=b\,e_{n-2}^{2}. \qedhere
		\]
	\end{proof}
	
	Thus in every dimension the BT data are, verbatim, the planar formulas
	\eqref{eq:planar} with $\dt$ and $\tr$ replaced by $\delta$ and $\tau$.
	The name is justified by the following, and by nothing stronger: at $p$
	the numbers $\delta,\tau$ and their first derivatives along $\ker J$
	coincide with the determinant, the trace, and their first derivatives for
	the Jacobian of the reduced field on a centre manifold, because both
	pairs vanish at $p$ and \eqref{eq:two} pins the derivatives.  Away from
	$\ker J$, and at higher order, $\delta$ and $\tau$ are \emph{not} the
	determinant and trace of any invariantly defined $2\times2$ block; they
	are the two rational functions of the principal minors whose $1$-jet
	along the kernel line reproduces $-2a$ and $b$.
	
	\begin{corollary}[Nondegeneracy as transversality]\label{cor:transv}
		Let $S_{\delta}:=\{x\in U':\delta(x)=0\}=\{x\in U':\dt DX(x)=0\}$ and
		$S_{\tau}:=\{x\in U':\tau(x)=0\}$, and let
		$\ell:=p+\R q_{0}$ be the kernel line.  Then
		\[
		a\neq0\iff \ell\pitchfork_{p} S_{\delta},
		\qquad
		b\neq0\iff \ell\pitchfork_{p} S_{\tau},
		\]
		and in either case the corresponding set is a $C^{1}$ hypersurface
		near $p$.
	\end{corollary}
	
	\begin{proof}
		$D_{q_{0}}\delta=-2a$ and $D_{q_{0}}\tau=b$ by
		Proposition~\ref{prop:geom}; a nonzero directional derivative of a
		$C^{1}$ function along $q_{0}$ is exactly the statement that
		$d\delta_{p}\neq0$ with $q_{0}\notin\ker d\delta_{p}$, i.e.\ that the
		level set is a hypersurface transverse to $\ell$ at $p$.
	\end{proof}
	
	Note that $a,b\neq0$ gives transversality of $\ell$ to each hypersurface
	separately; it does not by itself make $S_{\delta}$ and $S_{\tau}$
	transverse to each other.
	
	\begin{corollary}[Topological type of the versal unfolding]\label{cor:sign}
		Let $X\in C^{3}$ satisfy the hypotheses of Corollary~\ref{cor:test}.
		Then
		\begin{equation}\label{eq:abprod}
			ab=-\frac{D_{q_{0}}e_{n}\bigl(e_{n-2}D_{q_{0}}e_{n-1}
				-e_{n-3}D_{q_{0}}e_{n}\bigr)}{2\,e_{n-2}^{3}},
		\end{equation}
		the versal unfolding of the restriction of $X$ to a centre manifold is
		smoothly equivalent to
		\begin{equation}\label{eq:versal}
			\dot\eta_{1}=\eta_{2},
			\qquad
			\dot\eta_{2}=\beta_{1}+\beta_{2}\eta_{1}+\eta_{1}^{2}
			+s\,\eta_{1}\eta_{2},
			\qquad s=\operatorname{sign}(ab)\in\{\pm1\},
		\end{equation}
		and
		\begin{equation}\label{eq:signab}
			s=-\operatorname{sign}\Bigl[
			e_{n-2}\;D_{q_{0}}e_{n}\;
			\bigl(e_{n-2}D_{q_{0}}e_{n-1}-e_{n-3}D_{q_{0}}e_{n}\bigr)\Bigr].
		\end{equation}
	\end{corollary}
	
	\begin{proof}
		\eqref{eq:abprod} is the product of the two formulas
		\eqref{eq:mainbis}.  The reduction of \eqref{eq:BT} with $ab\neq0$ to
		\eqref{eq:versal} with $s=\operatorname{sign}(ab)$ is classical
		\cite[Lemma~8.10 and Thm.~8.4]{Kuznetsov2004}.  Since
		$\operatorname{sign}(e_{n-2}^{3})=\operatorname{sign}(e_{n-2})$,
		\eqref{eq:abprod} gives \eqref{eq:signab}.
	\end{proof}
	
	\begin{remark}[Scope of the invariance]\label{rem:scope}
		Neither $a$, nor $b$, nor the product $ab$ is a coordinate-free number:
		by Lemma~\ref{lem:welldef}(b) the substitution $q_{0}\mapsto cq_{0}$
		sends $(a,b)$ to $(ca,cb)$ and $ab$ to $c^{2}ab$, and the passage from
		\eqref{eq:BT} to \eqref{eq:versal} rescales them further.  What is
		claimed here is exactly this: the right-hand side of
		\eqref{eq:signab} is unchanged (i) under $q_{0}\mapsto cq_{0}$, since
		the bracket is then multiplied by $c^{2}>0$, and (ii) under any
		$C^{2}$ change of phase-space coordinates with $q_{0}$ transported as
		in Lemma~\ref{lem:inv}, since $e_{n-2},e_{n-3}$ are similarity
		invariants of $J$ and the two directional derivatives are invariant by
		that lemma.  Under those transformations $s$ is well defined, and it
		is the standard topological modulus of the versal Bogdanov--Takens
		unfolding.  We make no claim about orbital equivalence or time
		reparametrizations, under which $a$ and $b$ scale differently.
	\end{remark}
	
	\begin{remark}[What each identity controls]\label{rem:controls}
		The roles of $a$ and $b$ are read off from the unfolding
		\emph{before} the rescaling that produces \eqref{eq:versal}, namely
		\begin{equation}\label{eq:unfold}
			\dot\eta_{1}=\eta_{2},
			\qquad
			\dot\eta_{2}=\beta_{1}+\beta_{2}\eta_{1}+a\,\eta_{1}^{2}
			+b\,\eta_{1}\eta_{2}+\mathcal O(|\eta|^{3}),
		\end{equation}
		in the coordinates of Lemma~\ref{lem:identify}.  Its equilibria are
		$\eta_{2}=0$, $\beta_{1}+\beta_{2}\eta_{1}+a\eta_{1}^{2}=0$, and there
		\[
		J_{c}=\begin{pmatrix}0&1\\ \beta_{2}+2a\eta_{1}& b\eta_{1}\end{pmatrix},
		\qquad
		\tr J_{c}=b\,\eta_{1},
		\qquad
		\dt J_{c}=-(\beta_{2}+2a\eta_{1}).
		\]
		Three consequences.
		\begin{enumerate}[label=(\roman*),leftmargin=*]
			\item \emph{Curvature of the equilibrium branch.}  On the line
			$\beta_{2}=0$ the branch is $\beta_{1}=-a\eta_{1}^{2}$, so its
			curvature at the origin is
			$\partial^{2}\beta_{1}/\partial\eta_{1}^{2}=-2a
			=D_{q_{0}}e_{n}/e_{n-2}$.
			\item \emph{Saddle-node curve.}  Imposing $\dt J_{c}=0$ on the
			equilibrium set gives the parametrization
			$(\beta_{1},\beta_{2})=(a\eta_{1}^{2},-2a\eta_{1})$, i.e.\
			$\beta_{2}^{2}=4a\beta_{1}$; this is a nondegenerate cusp exactly
			when $a\neq0$, that is, exactly when $\ell\pitchfork S_{\delta}$ by
			Corollary~\ref{cor:transv}.
			\item \emph{Hopf half-line.}  Imposing $\tr J_{c}=0$ gives
			$b\eta_{1}=0$; if $b\neq0$ this forces $\eta_{1}=0$, hence
			$\beta_{1}=0$, and $\dt J_{c}=-\beta_{2}>0$ selects the half-line
			$\beta_{1}=0$, $\beta_{2}<0$.  If $b=0$ the trace does not vary
			along the branch at first order and no Hopf half-line emanates
			from the origin at this order.  So $b\neq0$, i.e.\
			$\ell\pitchfork S_{\tau}$, is what creates the Hopf branch.
		\end{enumerate}
		The numbers $-2a$ and $b$ appearing in (i)--(iii) are attached to the
		representation \eqref{eq:unfold}, which depends on the normalization of
		$q_{0}$ and on the chosen unfolding parameters; by
		Remark~\ref{rem:scope} it is their vanishing and the sign of their
		product that are intrinsic.  Finally, identity \eqref{eq:trace} says
		that the naive candidate $D_{q_{0}}\tr DX$ is the wrong quantity as
		soon as $n\ge3$: it exceeds $b$ by $\tr\Ltr$, and the second term
		$-e_{n-3}D_{q_{0}}e_{n}/e_{n-2}^{2}$ of \eqref{eq:mainbis} is exactly
		the correction that removes the transverse contribution.
	\end{remark}
	
	\section{What transfers to a Hopf point, and what does not}\label{sec:hopf}
	
	\subsection{The linear level}
	
	At a Hopf point the analogue of Theorem~\ref{thm:main} at the linear
	level is classical.  If $J$ has a pair $\pm i\omega_{0}$, the Hurwitz
	determinant $\Delta_{n-1}$ of $\chi_{J}$ satisfies Orlando's identity
	\cite{Orlando1911}
	\begin{equation}\label{eq:orlando}
		\Delta_{n-1}=(-1)^{\frac{n(n-1)}{2}}\prod_{i<j}(\mu_{i}+\mu_{j}),
	\end{equation}
	so $\Delta_{n-1}(J)=0$ is a \emph{necessary} condition for a purely
	imaginary pair; it is not sufficient, since $\Delta_{n-1}$ also vanishes
	whenever any two eigenvalues sum to zero, for instance a real transverse
	pair $\{\mu,-\mu\}$.  Writing
	$\mu_{1,2}=\tfrac12\tau_{c}\pm i\omega$ for the critical pair,
	\eqref{eq:orlando} factors as $\Delta_{n-1}=\tau_{c}\,\Omega$ with
	\begin{equation}\label{eq:Omega}
		\Omega=(-1)^{\frac{n(n-1)}{2}}(-1)^{\frac{(n-2)(n-3)}{2}}\,
		\Delta_{n-3}\bigl(\chitr\bigr)\;
		\Res_{\lambda}\bigl(\chitr(\lambda),\,
		\lambda^{2}-\tau_{c}\lambda+\delta_{c}\bigr) ;
	\end{equation}
	one checks this by splitting the product in \eqref{eq:orlando} into the
	pairs with both indices transverse, which give $\Delta_{n-3}(\chitr)$ by
	\eqref{eq:orlando} applied to $\chitr$, and the mixed pairs, whose
	product is $\prod_{j\ge3}(\mu_{j}^{2}-\tau_{c}\mu_{j}+\delta_{c})
	=\Res_{\lambda}(\chitr,\lambda^{2}-\tau_{c}\lambda+\delta_{c})$.  At the
	critical point $\Res=|\chitr(i\omega_{0})|^{2}$, which is nonzero exactly
	when $\pm i\omega_{0}$ is a simple pair, while
	$\Delta_{n-3}(\chitr)\neq0$ exactly when no two transverse eigenvalues
	sum to zero.  \emph{Both} conditions are needed: assuming only the
	absence of a second purely imaginary pair and of a zero eigenvalue is not
	enough.  Under them the transversality condition becomes
	$\frac{d}{d\theta}\operatorname{Re}\lambda
	=\bigl(2\Omega\bigr)^{-1}D_{\theta}\Delta_{n-1}$, for a one-parameter
	family with parameter $\theta$, which is the criterion of
	Liu \cite{Liu1994}; see also Kruff and Walcher
	\cite{KruffWalcher2020} for a coordinate-free treatment.
	
	\subsection{The nonlinear level: an obstruction of order, not of structure}
	
	Let $q\in\C^{n}$ and $\varphi\in\C^{n}$ be right and left eigenvectors
	of $J$ for the eigenvalue $i\omega_{0}$, that is $Jq=i\omega_{0}q$ and
	$J^{\top}\varphi=-i\omega_{0}\varphi$, normalized by
	$\ip{\varphi}{q}=1$, and let $B=D^{2}X(p)$ and $C=D^{3}X(p)$.  The first
	Lyapunov coefficient can then be written using directional derivatives of
	$DX$ regarded as a matrix-valued map,
	\begin{equation}\label{eq:l1}
		\ell_{1}=\frac{1}{2\omega_{0}}\operatorname{Re}
		\bigl\langle \varphi,\;
		(D_{\bar q}D_{q}J)q
		-2\,(D_{q}J)J^{-1}(D_{\bar q}J)q
		+(D_{\bar q}J)(2i\omega_{0}I-J)^{-1}(D_{q}J)q
		\bigr\rangle,
	\end{equation}
	but \eqref{eq:l1} is only a rewriting of the standard formula
	\cite[Ch.~5]{Kuznetsov2004} with $(D_{v}J)w=B(v,w)$ and
	$(D_{w}D_{v}J)u=C(v,w,u)$; it is not a formula in the scalar invariants
	$e_{k}$.  What can be asserted rigorously is the following.
	
	\begin{proposition}\label{prop:hopf}
		$\ell_{1}$ is not a function of $\bigl\{J,\ D_{v}e_{k}(DX)(p)\bigr\}$.
	\end{proposition}
	
	\begin{proof}
		$D_{v}e_{k}(DX)(p)$ depends only on the $2$-jet of $X$ at $p$, whereas
		$\ell_{1}$ depends on $D^{3}X(p)$.  Concretely, the planar family
		$\dot x=-y$, $\dot y=x+g\,x^{2}y$ has, for every $g$, the same $2$-jet
		at the origin, namely the linear rotation, while its first Lyapunov
		quantity equals $g/8$.
	\end{proof}
	
	Thus the programme of the present paper, which uses \emph{first-order}
	directional derivatives of the characteristic invariants, does not extend
	to $\ell_{1}$.  We stress that this is an obstruction of order and not of
	structure, and in particular the following is \emph{not} true: that a
	symmetric scalar invariant of $DX(x)$ cannot record an ordered
	composition of two copies of $B$ mediated by a resolvent.  It can, at
	second order.  Indeed, with $M=\lambda I-J$, $B_{v}:=B(v,\cdot)$ and
	$K=M^{-1}B_{v}$, the expansion
	$\dt(M-tB_{v})=\dt M\cdot\dt(I-tK)=\dt M\sum_{j}(-t)^{j}e_{j}(K)$
	gives
	\begin{equation}\label{eq:second}
		\left.\frac{d^{2}}{dt^{2}}\right|_{0}\chi_{DX(p+tv)}(\lambda)
		=\dt M\Bigl[(\tr K)^{2}-\tr(K^{2})\Bigr]
		+\tr\Bigl(\adj(M)\,D^{2}(DX)_{p}[v,v]\Bigr),
	\end{equation}
	and the term
	$\dt(\lambda I-J)\,\tr\bigl(((\lambda I-J)^{-1}B(v,\cdot))^{2}\bigr)$ is
	exactly such an ordered composition; knowing the polynomial
	\eqref{eq:second} for all $\lambda$ determines that rational function of
	$\lambda$ and hence its value at $\lambda=2i\omega_{0}$.  Whether some
	higher-order jet of the characteristic invariants determines $\ell_{1}$
	is an open question that we do not address.
	
	There is nevertheless a structural reason why the BT case is so clean and
	should not be expected to repeat.  At a BT point the relevant object is
	$\adj(\lambda I_{2}-N)$, whose constant term is the rank-one matrix
	$E_{12}$; contracting it against $L_{c}$ isolates the single entry
	$[L_{c}]_{21}=2a$, and the linear term isolates $\tr L_{c}=b$.  Both are
	first-order data of $x\mapsto DX(x)$, hence visible to a first derivative
	of a scalar invariant.  The BT coefficients are quadratic data attached
	to a nilpotent block; $\ell_{1}$ is not.
	
	\section{Algorithm and implementation}\label{sec:algo}
	
	Formula \eqref{eq:mainbis} suggests the following algorithm, in which $p$
	and possibly parameters are symbolic.
	
	\begin{enumerate}[label=\textbf{S\arabic*.},leftmargin=*]
		\item Compute $q_{0}$ spanning $\ker DX(p)$ and check
		$\rank DX(p)=n-1$.
		\item Substitute $x\mapsto p+s\,q_{0}$ into $J(x)=DX(x)$
		\emph{before} any further computation.  This turns a matrix of
		multivariate polynomials into a matrix of univariate polynomials in
		$s$.
		\item Compute $\chi(\lambda,s)=\dt(\lambda I-J(p+sq_{0}))$ once and
		truncate at order $\mathcal O(s^{2})$.
		\item Read $e_{n},e_{n-1},e_{n-2},e_{n-3}$ off the coefficient list of
		$\chi$ in $\lambda$, with the convention $e_{k}=0$ for $k<0$, and
		apply \eqref{eq:mainbis}.
	\end{enumerate}
	
	Three remarks on cost.  First, only four of the $e_{k}$ are ever needed;
	computing them from the definition as sums of principal minors costs
	$1+n+\binom n2+\binom n3=O(n^{3})$ determinants of size at most $n$,
	whereas one characteristic polynomial delivers all of them at once.  Over
	a polynomial ring a division-free method such as the Berkowitz algorithm
	computes $\chi$ in $O(n^{4})$ ring operations while avoiding the
	denominators introduced by Faddeev--LeVerrier or by fraction-free
	elimination; we do not make any claim about which method a particular
	computer algebra system uses internally.  Second, truncating in $s$ at
	S2--S3 rather than after expansion is what prevents intermediate
	expression swell: the BT data are a first-order jet, and there is no
	reason to build the full multivariate characteristic polynomial.  Third,
	in contrast with \eqref{eq:abKuz}, no null vector of $J^{\top}$ and no
	generalized eigenvector $q_{1}$ is required, so two linear solves over
	the field of rational functions in the parameters are avoided.
	
	A minimal implementation is the following.  Two details are easy to get
	wrong.  The guard \texttt{If[k<0,0,\ldots]} is what makes the case $n=2$,
	in which $k=n-3=-1$, avoid an invalid part specification.  And the
	characteristic polynomial has to be formed as a determinant, because the
	built-in \textsc{Wolfram Language} function returns $\dt(J-\lambda I)$,
	which is $(-1)^{n}$ times the convention \eqref{eq:chi} used here; using
	it would flip the sign of every $e_{k}$ for odd $n$.
	
	\begin{verbatim}
		BTCoefficients[X_List, vars_List, p_List, q0_List] :=
		Module[{n = Length[vars], s, lam, J, chi, c, e},
		J = D[X, {vars}] /. Thread[vars -> p + s q0];
		chi = Det[lam IdentityMatrix[n] - J];      (* see the caveat above *)
		chi = Normal@Series[chi, {s, 0, 1}];
		c = PadRight[CoefficientList[chi, lam], n + 1]; (* c[[j+1]] = coeff lam^j *)
		e[k_] := If[k < 0, 0, (-1)^k c[[n - k + 1]]];   (* e_k(J(p + s q0))       *)
		With[{e2 = e[n - 2] /. s -> 0, e3 = e[n - 3] /. s -> 0,
			Dn = D[e[n], s] /. s -> 0, Dn1 = D[e[n - 1], s] /. s -> 0},
		{-Dn/(2 e2), Dn1/e2 - e3 Dn/e2^2}]]
	\end{verbatim}
	
	The accompanying file \texttt{Supplementary\-Verification.wl} provides
	a self-contained symbolic verification of the results, including the
	nondegeneracy test of Corollary~\ref{cor:test}, the central functions
	\eqref{eq:deltatau}, the sign \eqref{eq:signab}, a direct verification
	of the generating identity \eqref{eq:gen}, and a cross-validation of
	the coefficients against \eqref{eq:abKuz}; see Section~\ref{sec:supp}.
	\section{Examples}\label{sec:examples}
	
	\subsection{The generic three-dimensional germ}
	
	Let $n=3$ and let $X$ be written in coordinates $(y_{1},y_{2},z)$ in
	which $J=DX(0)$ is already adapted:
	\begin{equation}\label{eq:ex1}
		\begin{aligned}
			\dot y_{1}&=y_{2}+\alpha_{20}y_{1}^{2}+\alpha_{11}y_{1}y_{2}
			+\alpha_{02}y_{2}^{2}+c_{1}y_{1}z+c_{2}y_{2}z+c_{3}z^{2},\\
			\dot y_{2}&=\beta_{20}y_{1}^{2}+\beta_{11}y_{1}y_{2}
			+\beta_{02}y_{2}^{2}+d_{1}y_{1}z+d_{2}y_{2}z+d_{3}z^{2},\\
			\dot z&=\kappa z+g_{20}y_{1}^{2}+g_{11}y_{1}y_{2}+g_{02}y_{2}^{2}
			+h_{1}y_{1}z+h_{2}y_{2}z+h_{3}z^{2},
		\end{aligned}
	\end{equation}
	with $\kappa\neq0$.  Here $q_{0}=(1,0,0)$, $e_{n-2}=e_{1}(J)=\kappa$ and
	$e_{n-3}=e_{0}=1$.  Evaluating $DX$ at $x=s\,q_{0}$ and keeping
	first-order terms,
	\[
	D_{q_{0}}e_{3}=-2\kappa\beta_{20},
	\qquad
	D_{q_{0}}e_{2}=\kappa\bigl(2\alpha_{20}+\beta_{11}\bigr)-2\beta_{20},
	\qquad
	D_{q_{0}}e_{1}=2\alpha_{20}+\beta_{11}+h_{1},
	\]
	so that Theorem~\ref{thm:main} gives
	\begin{equation}\label{eq:ex1ab}
		a=-\frac{1}{2}\cdot\frac{-2\kappa\beta_{20}}{\kappa}=\beta_{20},
		\qquad
		b=\frac{\kappa(2\alpha_{20}+\beta_{11})-2\beta_{20}}{\kappa}
		-\frac{1\cdot(-2\kappa\beta_{20})}{\kappa^{2}}
		=2\alpha_{20}+\beta_{11},
	\end{equation}
	in agreement with Lemma~\ref{lem:identify}.  Two structural facts are
	visible.  None of the twelve transverse coupling coefficients
	$c_{i},d_{i},g_{i},h_{i}$ occurs in \eqref{eq:ex1ab}: the entire third
	equation and every $z$-dependent term of the first two are invisible to
	$a$ and $b$, as they must be by Step~1 of Lemma~\ref{lem:identify}.  By
	contrast $h_{1}$ does occur in $D_{q_{0}}e_{1}$, in accordance with
	\eqref{eq:trace}, where $\tr\Ltr=h_{1}$.
	
	\subsection{A worked instance}
	
	Consider on $\R^{3}$
	\begin{equation}\label{eq:ex2}
		\dot x=y+x^{2}+xz,
		\qquad
		\dot y=3x^{2}-xy+yz+z^{2},
		\qquad
		\dot z=-z+x^{2}+2xy-y^{2}+xz .
	\end{equation}
	The origin is an equilibrium with
	\[
	J=\begin{pmatrix}0&1&0\\0&0&0\\0&0&-1\end{pmatrix},
	\qquad
	e_{1}=-1,\quad e_{2}=e_{3}=0,\quad \rank J=2,
	\]
	so $(\mathrm{BT}_{0})$ holds and $\ker J=\langle q_{0}\rangle$ with
	$q_{0}=(1,0,0)$.  Along $x=(s,0,0)$,
	\[
	e_{1}(s)=2s-1,\qquad
	e_{2}(s)=-7s-3s^{2},\qquad
	e_{3}(s)=6s-4s^{2}+12s^{3},
	\]
	hence $D_{q_{0}}e_{3}=6$, $D_{q_{0}}e_{2}=-7$, $D_{q_{0}}e_{1}=2$, and by
	\eqref{eq:mainbis} with $e_{1}=-1$, $e_{0}=1$,
	\[
	a=-\tfrac12\cdot\frac{6}{-1}=3,
	\qquad
	b=\frac{-7}{-1}-\frac{1\cdot6}{(-1)^{2}}=1 .
	\]
	Direct evaluation of \eqref{eq:abKuz} with $q_{1}=(0,1,0)$,
	$p_{0}=(1,0,0)$, $p_{1}=(0,1,0)$ gives $B(q_{0},q_{0})=(2,6,2)$ and
	$B(q_{0},q_{1})=(0,-1,2)$, whence $a=\tfrac12\cdot6=3$ and
	$b=2+(-1)=1$, in agreement.  Since $\Jtr=(-1)$ is hyperbolic and
	$ab=3\neq0$, the origin is a nondegenerate BT singularity, and by
	\eqref{eq:signab}
	\[
	s=-\operatorname{sign}\bigl[(-1)\cdot6\cdot((-1)(-7)-1\cdot6)\bigr]
	=-\operatorname{sign}(-6)=+1 .
	\]
	Finally $\tr\Ltr=1$ and indeed $D_{q_{0}}e_{1}=2=b+1$, illustrating
	\eqref{eq:trace}.
	
	\subsection{A non-hyperbolic transverse block}
	
	\begin{example}\label{ex:nonhyp}
		Take $n=5$ and $J=N\oplus\Jtr$ with
		$\Jtr=\operatorname{diag}\bigl(1\bigr)\oplus
		\left(\begin{smallmatrix}0&-2\\2&0\end{smallmatrix}\right)$, so that
		$\spec\Jtr=\{1,\pm2i\}$ and $\dt\Jtr=4\neq0$, but $\Jtr$ is not
		hyperbolic and there need not exist an invariant two-dimensional
		manifold tangent to $\Ec$.  For any quadratic $X$ with $DX(0)=J$,
		Theorem~\ref{thm:main} still holds and reproduces \eqref{eq:abKuz}
		exactly; only Lemma~\ref{lem:identify} is unavailable.  This is
		verified in the accompanying script.
	\end{example}
	
	\subsection{A parameter-dependent test from the literature}
	
	\begin{example}\label{ex:diasmello}
		Dias and Mello \cite{DiasMello2010} study the three-dimensional
		quadratic family
		\begin{equation}\label{eq:dm}
			\dot x=y,\qquad \dot y=z,\qquad
			\dot z=-\bigl[(a_{1}x+a_{0})z+(b_{1}x+b_{0})y+x^{2}+c_{0}\bigr],
		\end{equation}
		and prove in their Theorem~4.16 that for $b_{0}=c_{0}=0$, $a_{0}\neq0$,
		$b_{1}\neq2/a_{0}$ and $a_{1}\in\R$ arbitrary the origin is a
		nondegenerate BT point with
		\begin{equation}\label{eq:dmab}
			a=-\frac{1}{a_{0}^{2}},\qquad b=\frac{2-a_{0}b_{1}}{a_{0}^{3}},
		\end{equation}
		obtained from \eqref{eq:abKuz} with the chains
		$q_{0}=(1/a_{0},0,0)$, $q_{1}=(0,1/a_{0},0)$,
		$p_{0}=(a_{0},0,-1/a_{0})$, $p_{1}=(0,a_{0},1)$.  Theorem~\ref{thm:main}
		recovers this without any of those four vectors.  At the origin
		\[
		J=\begin{pmatrix}0&1&0\\0&0&1\\0&0&-a_{0}\end{pmatrix},
		\qquad
		\rank J=2,\quad e_{2}(J)=0,\quad e_{1}(J)=-a_{0}\neq0,
		\]
		so $(\mathrm{BT}_{0})$ holds and $\ker J=\langle(1,0,0)\rangle$.  Along
		$x=(s,0,0)$,
		\[
		e_{1}(s)=-a_{0}-a_{1}s,\qquad e_{2}(s)=b_{1}s,\qquad e_{3}(s)=-2s,
		\]
		whence $D_{q_{0}}e_{3}=-2$, $D_{q_{0}}e_{2}=b_{1}$,
		$D_{q_{0}}e_{1}=-a_{1}$, and \eqref{eq:mainbis} with $e_{1}=-a_{0}$,
		$e_{0}=1$ gives
		\[
		a=-\frac12\cdot\frac{-2}{-a_{0}}=-\frac{1}{a_{0}},
		\qquad
		b=\frac{b_{1}}{-a_{0}}-\frac{1\cdot(-2)}{a_{0}^{2}}
		=\frac{2-a_{0}b_{1}}{a_{0}^{2}} .
		\]
		These are $a_{0}$ times the published values \eqref{eq:dmab}, as
		predicted by the scaling law of Lemma~\ref{lem:welldef}(b), because
		the normalized generator $(1,0,0)$ is $a_{0}$ times the generator
		$(1/a_{0},0,0)$ used in \cite{DiasMello2010}.  The factor is $a_{0}$
		and not $a_{0}^{2}$ because the whole adapted frame moves at once:
		$q_{0}\mapsto cq_{0}$ forces $q_{1}\mapsto cq_{1}$ and hence
		$p_{0}\mapsto c^{-1}p_{0}$, $p_{1}\mapsto c^{-1}p_{1}$, so that the
		expressions \eqref{eq:abKuz}, being homogeneous of degree two in the
		$q$'s and of degree $-1$ in the $p$'s, are homogeneous of degree
		\emph{one} in $c$.  Explicitly, with $c=a_{0}$ the four vectors of
		\cite{DiasMello2010} become
		\[
		q_{0}'=(1,0,0),\quad q_{1}'=(0,1,0),\quad
		p_{0}'=\bigl(1,0,-a_{0}^{-2}\bigr),\quad
		p_{1}'=\bigl(0,1,a_{0}^{-1}\bigr),
		\]
		which again satisfy $Jq_{0}'=0$, $Jq_{1}'=q_{0}'$,
		$J^{\top}p_{1}'=0$, $J^{\top}p_{0}'=p_{1}'$ and \eqref{eq:dual}; and
		since $B(q_{0}',q_{0}')=(0,0,-2)$ and $B(q_{0}',q_{1}')=(0,0,-b_{1})$,
		\[
		\tfrac12\ip{p_{1}'}{B(q_{0}',q_{0}')}=-\frac{1}{a_{0}},
		\qquad
		\ip{p_{0}'}{B(q_{0}',q_{0}')}+\ip{p_{1}'}{B(q_{0}',q_{1}')}
		=\frac{2}{a_{0}^{2}}-\frac{b_{1}}{a_{0}}
		=\frac{2-a_{0}b_{1}}{a_{0}^{2}},
		\]
		in agreement with \eqref{eq:mainbis}.  The two computations therefore
		do not merely agree up to an unexplained constant: they agree on the
		nose once the same generator is used, and the ratio between them is
		precisely the one the covariance law predicts.  In particular the
		quantities that do not depend on the normalization, namely the
		vanishing of $a$ and of $b$ and the sign of $ab$, are literally the
		same on both sides.  The nondegeneracy conditions
		agree with theirs: $a\neq0$ for every $a_{0}\neq0$, and $b\neq0$ if and
		only if $b_{1}\neq2/a_{0}$.  Finally \eqref{eq:signab} returns
		\[
		s=-\operatorname{sign}\bigl[(-a_{0})(-2)\bigl((-a_{0})b_{1}+2\bigr)\bigr]
		=-\operatorname{sign}\bigl[2a_{0}(2-a_{0}b_{1})\bigr]
		=\operatorname{sign}\Bigl(b_{1}-\frac{2}{a_{0}}\Bigr),
		\]
		which is the value they report.  Note that $a_{1}$ occurs in
		$D_{q_{0}}e_{1}$ but in neither $a$ nor $b$, in accordance with Step~1
		of Lemma~\ref{lem:identify} and with \eqref{eq:trace}.  The whole
		computation is four principal minors and one null vector, with the four
		parameters carried symbolically throughout.
	\end{example}
	
	\section{Conclusions}\label{sec:conclusions}
	
	The two identities \eqref{eq:mainbis} say that the Bogdanov--Takens
	normal-form data of an equilibrium in $\R^{n}$ are the first-order
	variation, along the kernel of the linearization, of the two lowest
	coefficients of the characteristic polynomial, normalized by the
	determinant of the transverse block.  They are the coefficients of
	$\lambda^{0}$ and $\lambda^{1}$ of one polynomial identity,
	Theorem~\ref{thm:gen}, which describes the entire first-order spectral
	jet along $\ker J$ and shows at the same time exactly where the
	transverse block begins to interfere: by Theorem~\ref{thm:sharp} the
	invariants $e_{n-1}$ and $e_{n}$ are the only two whose first-order jet
	along $\ker J$ is free of transverse data, and the proof of that
	statement reduces to a gap lemma for polynomials, Lemma~\ref{lem:gap},
	whose sharp form may be of interest on its own.  Equivalently, by Proposition~\ref{prop:geom}, the
	planar formulas hold verbatim in every dimension once the determinant and
	the trace are replaced by the central determinant and the central trace
	\eqref{eq:deltatau}, and then the nondegeneracy conditions are the
	transversality of the kernel line to two explicit hypersurfaces of phase
	space and the topological type of the versal unfolding is the sign
	\eqref{eq:signab}.
	
	Three directions seem worth pursuing.  First, the degenerate cases
	$a=0$ or $b=0$ are governed by higher-order normal-form coefficients
	\cite{DRS1987,Kuznetsov2005}; the natural question is whether the
	second-order jet of $e_{n-1}$ and $e_{n}$ along $\ker J$ carries the cusp
	coefficient, and \eqref{eq:second} shows the shape such a formula would
	have.  Second, the combinatorial obstruction obtained in the planar
	Kolmogorov setting \cite{Companion}, where the mixed volume of the Newton
	polytopes bounds the order of the Takens normal form, has an evident
	$n$-dimensional analogue once $a$ and $b$ are expressed through
	$\dt DX$ and $e_{n-1}(DX)$, which are again Laurent polynomials on the
	torus.  Third, the elimination-theoretic reading of
	Corollary~\ref{cor:test} --- the BT locus is the saturation by $e_{n-2}$
	of the ideal generated by $X$, $e_{n-1}(DX)$, $e_{n}(DX)$ and the
	$(n-1)$-minors of $DX$ --- makes the systematic search for BT points in
	parametrized food-chain and reaction-network models a routine Gr\"obner
	computation, which the classical formulation in terms of adjoint
	eigenvectors does not.
	
	\section{Supplementary material}\label{sec:supp}
	
	A single file accompanies this paper, and it is deliberately
	self-contained: \texttt{Supplementary\-Verification.wl} is a
	\textsc{Wolfram Language} script that loads no package and requires
	nothing beyond a kernel, so that it can be read as a document and run as
	a certificate at once.  It prints a summary table of every assertion.
	
	Each assertion recomputes its reference value independently of the
	formula under test, so that the two sides of every comparison are
	obtained by different routes: the elementary invariants are recomputed
	as sums of principal minors, straight from \eqref{eq:chi}; the pair
	$(a,b)$ is cross-checked against the bilinear formula \eqref{eq:abKuz};
	the two sides of the generating identity \eqref{eq:gen} are built
	separately, the left from the jet of the invariants and the right from
	the transverse block in an adapted basis; and the planar normal form is
	recomputed by solving the homological equation of Step~2 of
	Lemma~\ref{lem:identify}.  The characteristic polynomial is formed
	explicitly as $\dt(\lambda I-A)$ rather than by the built-in function of
	the \textsc{Wolfram Language}, which returns $\dt(A-\lambda I)$ and
	would therefore reverse the sign of every $e_{k}$ for odd $n$.
	
	Its sections follow the paper and cover: the invariants against
	principal minors; the detection of $(\mathrm{BT}_{0})$ together with the
	two negative controls of Lemma~\ref{lem:split}; the generating identity
	\eqref{eq:gen} in dimensions $3\le n\le5$; Theorem~\ref{thm:main}
	against \eqref{eq:abKuz} in dimensions $2\le n\le6$, over transverse
	spectra that are real, complex, purely imaginary and $1{:}{-1}$
	resonant; the generic germ \eqref{eq:ex1} with all nineteen coefficients
	symbolic; the worked instance \eqref{eq:ex2}; the covariance of
	Lemma~\ref{lem:welldef}(b); the invariance of Lemma~\ref{lem:inv} under
	linear and nonlinear conjugation, together with the failure recorded in
	Remark~\ref{rem:necessary}; the planar specialization; Example~\ref{ex:nonhyp};
	the trace identity \eqref{eq:trace}; Proposition~\ref{prop:geom};
	Corollary~\ref{cor:sign}; the unfolding computation of
	Remark~\ref{rem:controls}; Proposition~\ref{prop:sharp}; the recurrence
	underlying Lemma~\ref{lem:gap}, the gap lemma itself and its sharpness,
	and Theorem~\ref{thm:sharp} on eight hostile spectra, among them a
	single Jordan block, the eighth roots of unity, which realize the
	extremal case of the gap lemma, and a $1{:}{-1}$ resonance; the order
	obstruction and the second-order identity \eqref{eq:second} of
	Section~\ref{sec:hopf}; and Example~\ref{ex:diasmello}, where the four
	parameters of \eqref{eq:dm} are carried symbolically and the published
	values \eqref{eq:dmab} are reproduced together with their nondegeneracy
	conditions and their sign.  All 101 assertions are reported as verified.
	
	The functions used by the script are also available, in a form meant for
	reuse rather than for reading, as a \textsc{Wolfram Language} paclet;
	that packaging is independent of this paper and is not needed to
	reproduce anything stated here.
	
	\section*{Acknowledgements}
	E. Chan-L\'opez acknowledges support from SECIHTI through the	``Estancias Posdoctorales por M\'exico'' program (CVU 422090).
	
	\section*{Ethics declarations}
	\subsection*{Conflict of interest}
	The author declares no conflict of interest.
	
	\subsection*{Data availability}
	No datasets were generated or analysed during this study.  The
	\textsc{Wolfram Language} verification script described in
	Section~\ref{sec:supp} is provided as supplementary material with this
	article; it is self-contained and requires no further software.
	

\end{document}